\documentclass[12pt]{article}
\usepackage{amsmath,amsthm,amssymb,color}
\usepackage{amssymb}
\usepackage{mathrsfs}
\usepackage{graphicx}
\usepackage{exscale}
\usepackage{relsize}
\newtheorem{theorem}{Theorem}[section] 
\newtheorem{remark}{Remark}[section] 
\newtheorem{lemma} {Lemma}[section] 
\newtheorem{proposition}[theorem]{Proposition}
\newtheorem{corollary}[theorem]{Corollary} 
\newtheorem{example}{Example}[section]

\title{Some results on optimal stopping problems for one-dimensional regular diffusions}
\author{Dongchao Huang and Jian Song}
\date{}
\begin{document}
\maketitle

\begin{abstract}
For a type of employee stock option (ESO) and an American put option with a barrier, we obtain closed-form formulae for the value functions and provide a complete characterization for optimal stopping/continuation regions. Some comparison principles for the critical levels and the value functions are  given. This work is inspired by the characterization of the value functions for general one-dimensional regular diffusion processes developed  in \cite{DK03} by Dayanik and Karatzas.
\end{abstract}


\vspace{1mm}


\section{Introduction}

Let $(\Omega, \mathcal{F},\mathbb{P})$ be a complete probability space, and $B=(B_t)_{t\geq 0}$ be a standard one-dimensional Brownian motion adapted to the filtration $(\mathcal{F}_t)_{t\geq 0}$. In $(\Omega, \mathcal{F},\mathbb{P})$, we consider a price process $X$ with the state space $\mathcal{I}\triangleq(0,+\infty)$ governed by 
\begin{align}\label{eq-X}
dX_t=\mu(X_t)dt+\sigma(X_t) dB_t,~~ X_0=x\in\mathcal I.
\end{align}
{\it Throughout this article, we shall make the following assumption for the diffusion $X$. }

\noindent{\bf Assumption A.} 
\begin{enumerate}
\item[(i)]  $\mu:\mathcal I \to\mathbb R$ and $\sigma: \mathcal I \to (0,+\infty)$ are measurable functions such that  SDE \eqref{eq-X} has a unique strong solution.
\item[(ii)] The function 
\begin{equation}\label{eq-theta}
\theta(x) \triangleq rx-\mu(x)
\end{equation}
 is  non-decreasing in $\mathcal I$. 
\item[(iii)]
The diffusion $X$ is {\it regular}  in $\mathcal{I}$, and $0$ is a {\it natural boundary}.  
\end{enumerate}
The diffusion process given in \eqref{eq-X} under {\bf Assumption A} includes several popupar models of asset prices such as geometric Brownian motion, CEV process with $\beta\ge0$ (\cite{LM10}) and Cox-Ingersoll-Ross(CIR) process (with Feller condition).

In this article, we focus on two optimal stopping problems, which concern an investment with a transaction cost and a minimum guarantee (or an employee stock option) and an American put option with a barrier, respectively.

Consider the value function, 
 \begin{equation}\label{eso1'}
V(x)\triangleq \sup_{\tau \geq 0} E_x[e^{-r\tau}(\max \lbrace l,X_\tau \rbrace -K)^+], ~ x>0,
\end{equation}
where $l$ and $K$ are two positive constants with $l> K$, and $\tau$ is a $\mathcal F_t$-stopping time.   Here we use $E_x$ to denote $E[\cdot|X_0=x]$, and  similarly, we shall use  $P_x$ to denote $P(\cdot|X_0=x)$  in this article.

The  optimal stopping problem (\ref{eso1'}) can be interpreted as an investment problem with a transaction cost $K$ and a minimum guarantee $l$. Suppose that an investor holds a certain stock and he/she wants to profit from selling the stock at the cost of the transaction fee $K$. As a risk-averse investor, he/she believes that if the stock price is currently very low, then it may still remain at a relatively low level for a considerable amount of time. The investor could get away from such a situation safely if his/her stock price has a satisfactory minimum guarantee $l$ (a common hedging strategy to get such a guarantee is to buy put options). 
 Subject to the minimum guarantee $l$ and the transaction cost $K$, the investor faces an investment problem of finding the best selling time in order to maximize  his/her profit, which mathematically is the optimal stopping problem (\ref{eso1'}). 

The value function $V(x)$ in  (\ref{eso1'})  is equivalent to
\begin{equation}\label{eso2}
V(x)= \sup_{\tau \geq 0} E_x[e^{-r\tau}\max \lbrace (X_\tau-K)^+, l-K \rbrace ],
\end{equation}
and this can  be interpreted as the value function of an employee stock option (ESO), the holder of which has an additional choice of cash $l-K$ besides the stock option.  In \cite{XG01}, Guo and Shepp considered ESO pricing problems with the price process $X$ modelled by a geometric Brownian motion.

 The value function of an American put option with a barrier is given as follows. For $x\in (0,d)$,
\begin{equation}\label{vf1}
V(x)\triangleq \sup_{\tau \geq 0} E_x[e^{-r\tau} (q-X_{\tau})^+I_{\lbrace \tau<\tau_d \rbrace }], \ \ \  \mbox{with} \ \ \ \tau_d \triangleq \inf \lbrace t\geq 0: X_t = d \rbrace,
\end{equation}
where  $q $ is the strike price, $d\in (q,+\infty)$ is a pre-set barrier, and $\tau_d$ is the time when the option is ``knocked out". Note that in relation to the risk-neutral pricing, $V(x)$ given in \eqref{vf1} is the premium of an American barrier put option with a dividend yield  $\theta(x)/x$, where $\theta(x)$ is given in \eqref{eq-theta}.


In the Black-Scholes pricing framework, the pricing problem \eqref{vf1} was first considered by Karatzas and  Wang in \cite{KS00}. In their paper, they reduced the optimal stopping problem \eqref{vf1} to a variational inequality, and then obtained closed-form expressions for the value functions by solving the  variational inequality explicitly. In a slightly different direction, still in the Black-Scholes  framework, Dai and Kwok  in \cite{Dai04} presented an analytic valuation formula for  knock-in American options under a trigger clause and showed that the in-out barrier parity relation could be no longer obtained for American barrier options unlike the European counterparty. For more details on American barrier options, we refer to \cite{Gao00,Jun13,Jun15} and the references therein.

 A typical methodology of solving optimal stopping problems is to transform them into free-boundary problems (or variational inequalities). The so-called ``guess-and-verify" technique is used a lot to solve the free-boundary problems (see, e.g., \cite{XG01}, \cite{KS00}, \cite{Ott13}, \cite{Peskir06} and  \cite{Pham09}). More specifically, one first needs to guess the structure of the optimal stopping strategy (stop/continuation region) which is often an artful task. One may then solve the free-boundary problem by imposing conditions (such as smooth-fit/continuous-fit) on the boundaries of continuation and stopping regions. Finally, one needs to validate the function obtained from the previous step as a solution via direct verification (and check the optimality of the stopping strategy).

  However, when the underlying process $X$ is a general diffusion process given in  \eqref{eq-X} rather than a specific process such as a (geometric) Brownian motion, the structure of the optimal stop/continuation region may depend on the functions $\mu(x)$ and $\sigma(x)$, and  the above-mentioned approach would be much more challenging. In contrast to the ``guess-and-verify'' technique, our method  finds directly the mechanism that describes the structure of the stopping/continuation regions for general diffusion processes, provides closed-form formulae for the value functions and complete characterization for optimal stopping strategies (see Theorem \ref{thm3.1}, Remark \ref{fb} and Theorem \ref{value function for barrier put option with American feature}).

 It is of interest in financial models to study the properties of the value functions,  such as the convexity and the monotonicity in the model parameters (see, e.g., \cite{JT03, E04, Henderson05, Hobson10, Yin12}).
  Taking advantage of the formulae, we obtain some properties for the value fuctions, and then investigate the impact of the variations of the internal parameters, such as the volatility $\sigma$ and the drift $\mu$, and the external parameters, such as the interest rate $r$, the minimum guarantee $l$, the transaction cost $K$, the strike price $q$, and  the barrier $d$, on the  optimal strategy and the value function. As a result, some comparison principles are obtained (see Propositions \ref{external comparison principles under ESO problem}, \ref{internal comparison principles under ESO problem},  \ref{external comparison principles for pricing barrier option with American feature}, \ref{internal comparison principles for pricing barrier option with American feature}  and \ref{internal comparison principles for pricing standard American put option}). In the proof of the comparison principles, the properties of the value function play a critical role. We point out that the comparison principle with respect to the drift is not surprising because of the comparison principle for SDEs with respect to the drift, but it is not the case for the comparison principle with respect to the volatility due to the lack of the comparison principle for SDEs with respect to the volatility (see Remark \ref{remark3.6}). 
 
 This article is organized as follows. In Section 2, we provide some preliminaries on the diffusion $X$ given in \eqref{eq-X} and the value function. In Section 3 and Section 4, we deal with the optimal stopping problems \eqref{eso1'} and \eqref{vf1}, respectively.

\section{Some preliminaries on the diffusion and the value function}
 In this section, first we recall some preliminaries on the diffusion process given in \eqref{eq-X} and the value function $V(x)=\sup_{\tau}E_x[e^{-r\tau}h(X_\tau)]$ with some reward function $h$. We also refer to \cite[Section 4.6]{IM74}, \cite[Chapter II]{BS02} and \cite[Section 2]{DK03} for more details.

Let $\tau_\kappa$ be the first passage time of the diffusion process $X$ to level $\kappa$, i.e., 
$\tau_{\kappa}\triangleq\inf \lbrace t\geq 0: X_t=\kappa \rbrace.$ Then $E_x[e^{-r\tau_{\kappa}}]$ admits the following representations
\begin{equation}\label{eq-st}
E_x[e^{-r\tau_{\kappa}}]=
\begin{cases}
\dfrac{\psi (x)}{\psi (\kappa)} &\mbox{if $ x\leq \kappa$},\\

\dfrac{\varphi (x)}{\varphi (\kappa)}& \mbox{if $ x\geq \kappa$},
\end{cases}
\end{equation}
where $\psi$ (resp. $\varphi$) is a strictly increasing (resp. decreasing) solution to the differential equation 
\begin{equation}\label{eq-L}
(\mathcal{L}-r)u(x)=0,~~ x\in \mathcal{I},
\end{equation}
where
$$ \mathcal{L}\triangleq \frac{1}{2}\sigma^2(x)\frac{d^2}{dx^2} +\mu (x) \frac{d}{dx}$$ is the infinitesimal generator of $X$.  Noting that $0$ is a  natural boundary for $X$, by {\bf Assumption A}, we have $$\lim_{x\rightarrow 0} \psi(x)=0,~ \lim_{x\rightarrow 0} \varphi(x)=+\infty.$$ Since $+\infty \notin \mathcal{I}$, we also have  $\lim\limits_{x\rightarrow +\infty} \varphi(x)=0$ (see \cite{BS02} Page 19).  Moreover, $\varphi(x)$ and $\psi(x)$ are convex on $(0,+\infty)$ under {\bf Assumption A} and the following so-called {\it transversality} condition  (cf. \cite[Corollary 1]{AL03}):
 \begin{equation}\label{condition-tran}
  \lim_{t\rightarrow +\infty} E_x[e^{-rt}X_t]=0,~~ \text{ for all } x\in \mathcal I.
 \end{equation}
  Note that the convexity of $\psi$ is assumed in Section 3, and the convexity of $\varphi$ is assumed in Section 4. 

Now consider the value function \begin{align*}
V(x)=\sup_{\tau \geq 0} E_x[e^{-r\tau} h(X_\tau)], \ \ \ x\in \mathcal I
\end{align*}
where $h(x)$ is a reward function which is bounded on every compact subset of $\mathcal I$  and satisfies $\sup_{x\in \mathcal I}h(x)>0$.

 The major instrument that the methodology of this article relies on is Proposition 5.12 in \cite{DK03}, which gives the following representation for the value function, 
\begin{equation}\label{formula-V}
 V(x)=\varphi(x) W(F(x)), ~ x \in \mathcal I,
 \end{equation}
   where $W: [0,+\infty)\rightarrow \mathbb{R}$ is the smallest concave majorant of the function 
\begin{equation}\label{H}
H(y)\triangleq 
\begin{cases}
\left( \dfrac{h}{\varphi} \right) \circ F^{-1}(y), & \mbox{if} \ y >0\\
0, &  \mbox{if} \  y=0,
\end{cases}
\end{equation} with
\begin{equation}\label{F}
F(x)\triangleq \frac{\psi(x)}{\varphi(x)},~~ x\in \mathcal I. 
\end{equation}
Note that $F(x)$ is strictly increasing on $\mathcal I$, and $$\lim_{x\to0^+}F(x)=0,~ \lim_{x\to +\infty}F(x)=+\infty, ~\lim_{y\to0^+}H(y)=0.$$ 

Let $S$ be the scale function of the diffusion $X$, i.e., for arbitrary $c\in \mathcal I$,
\begin{equation}\label{scale}
S(x)=\int_c^x \exp \left(-\int_c^y \frac{2\mu(z)}{\sigma^2(z)}dz \right) dy.
\end{equation}
The generalized {\it Wronskian determinant}  of  $f$ and $g$ is defined as
\begin{equation} \label{W}
W(f,g) (x)\triangleq g(x)\frac{df(x)}{dS(x)}-f(x)\frac{dg(x)}{dS(x)}=\frac{g^2(x)}{S'(x)}\left(\frac{f(x)}{g(x)}\right)'.
\end{equation}
Note that $W(\psi, \varphi)$ is a positive constant. Indeed, the derivative  
\begin{align*}
\frac {dW(\psi,\varphi)(x)}{dx}= \dfrac{\left(\psi''\varphi-\varphi''\psi\right)S'+\left(\psi\varphi'-\varphi\psi'\right)S''}{(S')^2}(x)
\end{align*}
is zero, because  
$\left(\psi''\varphi-\varphi''\psi\right)(x)=\dfrac{2\mu (x)}{\sigma^2(x)}\left(\psi\varphi'-\varphi\psi'\right)$ and 
$S''(x)=-2\mu(x)/\sigma^2(x)\cdot S'(x)$.  The positivity of $W(\psi,\varphi)$ follows from the monotonicity and positivity of $\psi, \varphi$ and $S$.

We finish this section by introducing the following proposition, which provides  formulae for the derivatives of  $H(\cdot)$ given in \eqref{H}.

 \begin{proposition} \label{Prop-H}
Let $ H(\cdot)$ be defined in (\ref{H}) and  assume that $h \in C^2(\mathcal I \setminus \mathcal N$), where $\mathcal N$ is a finite set of points in $\mathcal I$, $h:\mathcal I\to \mathbb R$ is a twice differentiable function. Then, on $\mathcal I \setminus \mathcal N$
\begin{align}\label{H'}
 H'(F(x))=\frac{W(h,\varphi)(x)}{W(\psi,\varphi)},
\end{align}
and 
\begin{align}\label{H''}
   H''(F(x))=\frac{2}{\sigma^2(x)\varphi(x)\left(F'(x)\right)^2}\left[(\mathcal{L}-r)h\right](x),
\end{align}
for all $x\in \mathcal I$, 
where the function $F(\cdot)$ is defined in \eqref{F}, and the Wronskian $W(f,g)$ is defined in \eqref{W}.
\end{proposition}

\begin{proof}
 Denote $y=F(x)$, then $H(y)=h(x)/\varphi(x).$ By the chain rule, we have on  $y \in F(\mathcal I \setminus \mathcal N)$,
\begin{align*}
H'(y)=\left(\frac{h}{\varphi}\right)'(x)\bigg/F'(x)=\dfrac{d(h/\varphi)}{dF}(x),
\end{align*}
and similarly, 
\begin{align*}
H''(y)=\frac{dH'(y)}{dx}\bigg/F'(x)=\frac{dH'(y)}{dF}(x).
\end{align*}
Noting that $F=\psi/\varphi$, and by the definitions of $W(\psi, \varphi)$ and $W(h,\varphi)$ given in \eqref{W} and \eqref{G}, respectively,
\begin{align*}
&H'(y)=\dfrac{d(h/\varphi)}{dF}(x)=\dfrac{\left(d/{dS}\right)\left(h/\varphi\right)}{dF/dS}(x)\\
&=\dfrac{\left(dh/{dS}\right) \cdot \varphi-h \cdot \left(d\varphi/{dS}\right)}{\left(d\psi/{dS}\right) \cdot \varphi-\psi \cdot \left(d\varphi/{dS}\right)}(x)=\dfrac{W(h,\varphi)(x)}{W(\psi,\varphi)}\, . 
\end{align*}
Consequently, 
\begin{align*}
&H''(y)=\dfrac{d}{dF}\left( \dfrac{W(h,\varphi)}{W(\psi,\varphi)}\right)=\dfrac{1}{W(\psi,\varphi)} \cdot\dfrac{dW(h,\varphi)}{dF}\\
&=\dfrac{1}{W(\psi,\varphi)} \cdot \dfrac{d\left(\varphi(x) \frac{dh}{dS}(x) - h(x) \frac{d\varphi}{dS}(x)\right) }{dF}=\dfrac{1}{W(\psi,\varphi)} \cdot \dfrac{\left(\varphi\cdot \frac{h'}{S'}-h \cdot \frac{\varphi'}{S'}\right)'}{F'}\\
&=\dfrac{1}{W(\psi,\varphi)} \cdot
\dfrac{\varphi \cdot (h''S'-h'S'')-h \cdot (\varphi '' S'-\varphi ' S'')}{F'\cdot (S')^2}\\
&=\dfrac{1}{W(\psi,\varphi)\cdot F' \cdot S'} \cdot \left[(\varphi h''-h\varphi '')+(h\varphi '-h' \varphi) \frac{S''}{S'} \right]\\
&\stackrel{(a)}{=} \dfrac{1}{W(\psi,\varphi) \cdot F' \cdot S'} \cdot \left[(\varphi h''-h\varphi '')-\frac{2\mu(x)}{\sigma^2(x)}(h\varphi '-h' \varphi) \right]\\
&=  \dfrac{2}{W(\psi,\varphi) \cdot F' \cdot S' \cdot \sigma^2(x)}\left[\varphi \cdot \left(\frac{1}{2}\sigma^2(x)h''+\mu(x)h'\right)-h \cdot \left(\frac{1}{2}\sigma^2(x)\varphi''+\mu(x)\varphi'\right)\right]\\
&\stackrel{(b)}{=}  \dfrac{2\varphi}{W(\psi,\varphi) \cdot F' \cdot S' \cdot \sigma^2(x)} [(\mathcal{L}-r)h](x)\\
&\stackrel{(c)}{=} \dfrac{2}{\varphi(x) \cdot \sigma^2(x) \cdot \left(F'(x)\right)^2} [(\mathcal{L}-r)h](x),
\end{align*}
where the equality (a) follows from $S''(x)/S'(x)=-2\mu(x)/\sigma^2(x)$, (b) follows from the fact that $\varphi$ is a solution to $\mathcal{L}u=ru$, and (c) holds because of the equality $$\frac{W(\psi,\varphi) \cdot S'}{\varphi^2}=\left(\frac{\psi}{\varphi}\right)'=F'.$$
\end{proof}

\begin{remark}
In light of \eqref{H''},  the concavity/convexity of the function $H(\cdot)$ on $\mathcal I$ only depends on the sign of $(\mathcal{L}-r)h(x)$, noting that $\varphi(x)$, $\sigma^2(x)$, and $(F'(x))^2$ are all positive. 
\end{remark}

\begin{remark}
Obviously a linear combination of $\psi$ and $\varphi$ is still a solution to \eqref{eq-L}, but the inverse is not straightforward. Interestingly,  using Proposition \ref{Prop-H} we can provide an alternative proof as follows. 
\end{remark}

\begin{corollary}
Any solution to \eqref{eq-L} can be represented as a linear combination of $\psi(\cdot)$ and $\varphi(\cdot)$.
\end{corollary}

\begin{proof}
Suppose $h$ is a solution to \eqref{eq-L},  i.e., $ [(\mathcal{L}-r)h](x)=0$ for $x\in \mathcal I$. It follows that $H''(y) \equiv 0$ for $y\in \mathcal I$. And thus $H(y)=C_1y+C_2$ for some constants $C_1$ and $C_2$. Note that $H(y)=\left( \frac{h}{\varphi} \right) \circ F^{-1}(y)$. Consequently, $\left( \frac{h}{\varphi} \right) \circ F^{-1}(y)=C_1y+C_2$. Therefore, 
\begin{align*}
\dfrac{h(x)}{\varphi(x)}=C_1 F(x)+C_2=C_1 \dfrac{\psi(x)}{\varphi(x)}+C_2=\dfrac{C_1 \psi(x)+C_2 \varphi(x)}{\varphi(x)}, 
\end{align*}
and hence
$h(x)=C_1 \psi(x) +C_2 \varphi(x).$
\end{proof}


\section{On an employee stock option (ESO)}
In this section, we consider the  optimal stopping problem \eqref{eso1'}. A typical example  in corporate finance is the following employee stock option (ESO) problem:
\begin{equation}\label{avf1}
V(x,l,K)=\sup_{\tau \geq 0} E_x[e^{-r\tau}(Y_\tau -K)^+],
\end{equation}
where $K$ is the strike price,  $Y_t=\max \lbrace l,X_t \rbrace$, $l$ is a constant (slightly) bigger than $K$, and $X_t$ is  defined in \eqref{eq-X}. This problem is equivalent to
\begin{equation}\label{avf2}
V(x)=\sup_{\tau \geq 0} E_x[e^{-r\tau}(({X_\tau}-l)^+ +s)] \triangleq  \sup_{\tau \geq 0} E_x[e^{-r\tau} g(X_\tau)]
\end{equation}
where $s=l-K>0$ and the reward function is given as 
\begin{equation}\label{reward'}
g(x) \triangleq s+(x-l)^+.
\end{equation}
 Throughout this section, besides {\bf Assumption A}, we also assume that $\psi(x)$ is convex on $(0,+\infty)$. Note that a sufficient condition for the convexity of $\psi(x)$ (and $\varphi(x)$)  is the tranversality condition  \eqref{condition-tran}.

\begin{lemma} \label{lemma-3-1} For $g(x) =s+(x-l)^+$,
\begin{equation*}
L_0\triangleq \lim_{x\rightarrow 0^+} \dfrac{g(x)}{\varphi(x)}=0, \ \ \ \L_{+\infty}  \triangleq \lim_{x\rightarrow +\infty} \dfrac{g(x)}{\psi(x)} = \lim_{x\rightarrow +\infty} \dfrac{g'(x)}{\psi'(x)}<+\infty.
\end{equation*}
\end{lemma}
\begin{proof}
The first equation follows from the continuity of $g$ and the fact  $\lim\limits_{x\to 0}\varphi(x)=+\infty$. The second equation holds because of the L'Hospital's rule, and the limit is finite because $\psi$ is strictly increasing and convex.
\end{proof}

Thanks to Lemma \ref{lemma-3-1}, we may apply \cite[Proposition 5.12]{DK03} and get that $$V(x)=\varphi(x) \tilde{W} (F(x)),$$ where $ \tilde{W} :[0,+\infty) \rightarrow \mathbb{R}$ is the smallest non-negative concave majorant of 
\begin{equation}\label{G}
G(y)\triangleq 
\begin{cases}
\left( \dfrac{g}{\varphi} \right) \circ F^{-1}(y), & \mbox{if} \ y >0\\
0, &  \mbox{if} \  y=0.
\end{cases}
\end{equation}
This fact is the key ingredient in the proof of  Theorem \ref{thm3.1}.


To present easier arguments in the proofs, throughout this section, we just consider the reward function $g$ given in \eqref{reward'}. However,  all the results in this section can be extended to value functions with a general reward function $g$ that satisfies the following conditions.

\begin{enumerate}
\item
$g$ is strictly positive, non-decreasing, continuous on $[0,+\infty)$, and twice differentiable except on some positive point $n_g$ with $g'(n_g^+)>g'(n_g^-)$. 
\item $\lim\limits_{x\to+\infty}g(x)=+\infty.$
\item
$(\mathcal{L}-r)g(x)<0$ for all $x\in (0,n_g)$ and $(\mathcal{L}-r)g(x)$ is non-increasing on $(n_g,+\infty)$.
\item
$L_{+\infty} \triangleq \lim\limits_{x\rightarrow +\infty} \dfrac{g(x)}{\psi(x)}<+\infty$. (Assuming this condition directly, we do not need the assumption that $\psi$ is convex.)
\end{enumerate}

The properties of the function $G(\cdot)$ in \eqref{G} provided by the following lemma will play a key role in finding the formula for the value function. 
\begin{lemma} \label{lemma-G}
The function $G(\cdot)$ in \eqref{G} is continuous on $[0,+\infty)$,  twice differentiable on $(0,+\infty) \setminus \{ F(l)\}$,  and possesses the following properties:
\begin{enumerate}
\item[(i)] 
$G(y)$ is strictly increasing on $[0,+\infty)$.

\item[(ii)]
$G(y)$ is strictly concave on $(0,F(l))$. Moreover, there exists a unique point $x_g \in [l,+\infty]$ such that $G(y)$ is convex on $(F(l),F(x_g))$ while it is strictly concave on $(F(x_g),+\infty)$.  Here, we use the following convention: if $x_g=l$, then the interval $(F(l),F(x_g))$ reads as the empty set $\emptyset$ and $(F(x_g),+\infty)$ reads as $(F(l),+\infty)$;  if $x_g=+\infty$, then the interval $(F(l),F(x_g))$ reads as $(F(l),+\infty)$ and $(F(x_g),+\infty)$ reads as $\emptyset$.

\item[(iii)]
\begin{enumerate}
\item[a)]
$\lim\limits_{y\rightarrow F(l)^+} G'(y)> \lim\limits_{y\rightarrow F(l)^-} G'(y)$; 
\item[b)]
$\lim\limits_{y \rightarrow 0}\frac{1}{G'(y)}=0$.
\item[c)]
$\lim\limits_{y \rightarrow +\infty} G'(y)$ exists and  $\lim\limits_{y \rightarrow +\infty} G'(y)=L_{+\infty}$.
\end{enumerate} 

\end{enumerate} 
\end{lemma}
\begin{proof}
Clearly $G(\cdot)$ is continuous on $(0,+\infty)$ and twice differentiable on $(0,+\infty)\backslash\{F(l)\}$. It is also continuous at $0$, noting that $\lim\limits_{y\to0+} G(y)=\lim\limits_{x\to0+}\frac{g(x)}{\varphi(x)}=0.$  Now we prove the properties (i)-(iii) of $G$.

(i) By Proposition \ref{Prop-H}, $G'(F(x))=\frac{W(g,\varphi)}{W(\psi, \varphi)}$. The result follows from the fact that $W(\psi, \varphi)$ is a positive constant, and  $W(g,\varphi) \triangleq \varphi(x) \frac{dg}{dS}(x) - g(x) \frac{d\varphi}{dS}(x)>0$, since $g$ is a non-decreasing positive function, $\varphi$ is strictly decreasing, and $S$ is strictly increasing.

(ii) Note that 
\begin{align*}
(\mathcal{L}-r)g(x)&=\frac{1}{2}\sigma^2(x) g''(x)+\mu(x) g'(x)-rg(x)=
\begin{cases}
-rs, & \mbox{if} \ x<l;\\
rK-\theta(x), & \mbox{if} \ x>l.
\end{cases}
\end{align*}
Hence $G(\cdot)$ is strictly concave on $(0, F(l))$ by Proposition \ref{Prop-H}.

Note  that $(\mathcal{L}-r)g(x)$ is a non-increasing function on $(l,+\infty)$, and define
$$x_g \triangleq \sup \Big\{ a>l: (\mathcal{L}-r)g(x) \geq 0  \text{ on } (l,a) \text{ and } (\mathcal{L}-r)g(x)\leq 0 \text{ on  }(a,+\infty)\Big\},$$ where  we use the convention that $\sup \emptyset = l$. It follows immediately from Proposition \ref{Prop-H} that  if $x_g \in (l,+\infty)$, $G$ is convex on $(F(l),F(x_g))$ and strictly concave on $(F(x_g),+\infty)$, and if $x_g=+\infty$, $G$ is convex on $(F(l),+\infty)$. Finally, note that $x_g=l$ only if $(\mathcal L-r) g(l)\le 0$ and $(\mathcal L-r) g(l+\varepsilon)<0$ for all $\varepsilon>0$. Hence if $x_g=l$, $G$ is strictly concave on $(F(l),+\infty)$.

(iii) Note that the left derivative $g'(l^-)=0$ and the right derivative $g'(l^+)=1$. Then, straight forward calculation shows that $$\lim_{x \rightarrow l^-} S'(x) W(h,\varphi)=\lim_{x \rightarrow l^-} \varphi(x) g'(x) -g(x) \varphi'(x)=-s\varphi'(l)$$ and $$\lim_{x \rightarrow l^+} S'(x) W(h,\varphi)=\lim_{x \rightarrow l^+} \varphi(x) g'(x) -g(x) \varphi'(x)=\varphi(l)-s\cdot\varphi'(l)>-s\varphi'(l).$$ 
Therefore, $\lim_{y\rightarrow F(l)^+} G'(y)> \lim_{y\rightarrow F(l)^-} G'(y)$. \\
For part b), note that $\frac{1}{G'(y)}$ is positive, continuous, and furthermore, increasing  on $(0,F(l))$ as $G(y)$ is strictly concave on $(0, F(l))$. As a consequence, $\lim\limits_{y \rightarrow 0} \frac{1}{G'(y)}$  exists and satisfies
\begin{equation*}
\lim_{y \rightarrow 0} \dfrac{1}{G'(y)}=\lim_{y \rightarrow 0} \dfrac{F'}{\left({g}/{\varphi}\right)'}\circ F^{-1}(y)= \lim_{x \rightarrow 0} \dfrac{F'(x)}{\left({g}/{\varphi}\right)'(x)}=a \geq 0.
\end{equation*}
Note that $\lim\limits_{x \rightarrow 0} F(x)=0$, $\lim\limits_{x \rightarrow 0}\frac{g(x)}{\varphi(x)}=\lim_{x \rightarrow 0} \frac{s}{\varphi(x)}=0$, and $\lim\limits_{x \rightarrow 0} \frac{\psi(x)}{g(x)}=\lim\limits_{x \rightarrow 0} \frac{\psi(x)}{s}=0$. Since $\lim\limits_{y \rightarrow 0} \frac{1}{G'(y)}$ exists, we can apply L' H\^opital's rule to get
\begin{equation*}
0=\lim_{x \rightarrow 0} \dfrac{\psi(x)}{g(x)}=\lim_{x \rightarrow 0}\dfrac{F(x)}{\left({g}/{\varphi}\right)(x)}=\lim_{x \rightarrow 0}\dfrac{F'(x)}{\left({g}/{\varphi}\right)'(x)}=a,
\end{equation*}
which together with the previous equation leads to $$\lim_{y \rightarrow 0}\frac{1}{G'(y)}=0.$$

Finally, we show $$\lim_{y \rightarrow +\infty} G'(y)=L_{+\infty}.$$ 
In view of $(ii)$, there exists $M>0$ such that on $(M,+\infty)$, $G'$ is Monotone. Noting that $G'(y)\ge0$ for $y\ge0,$  to obtain the existence of $\lim\limits_{y \rightarrow +\infty} G'(y)$, it suffices to show $$\limsup_{y \rightarrow +\infty} G'(y)<\infty.$$  Indeed,
\begin{align*}
\limsup_{y \rightarrow +\infty} G'(y) \ &= \ \limsup_{x\rightarrow +\infty} \dfrac{\left(g/ \varphi \right)'}{F'}(x) \ = \ \limsup_{x\rightarrow +\infty} \dfrac{g'\varphi-g\varphi'}{\psi'\varphi-\psi\varphi'}(x)\\
& \leq \ \limsup_{x\rightarrow +\infty} \dfrac{g'\varphi}{\psi'\varphi-\psi\varphi'}(x)+\limsup_{x\rightarrow +\infty} \dfrac{-g\varphi'}{\psi'\varphi-\psi\varphi'}(x)\\
& \leq \ \limsup_{x\rightarrow +\infty} \dfrac{g'(x)}{\psi'(x)} + \limsup_{x\rightarrow +\infty} \dfrac{g(x)}{\psi(x)}\\
& = \ 2\cdot L_{+\infty}<+\infty.
\end{align*}
As a consequence, $\lim\limits_{y \rightarrow +\infty} G'(y)$ exists and
\begin{equation*}
\lim_{y\rightarrow +\infty}G'(y)= \lim_{y\rightarrow +\infty} \dfrac{\left({g}/{\varphi}\right)'}{F'}\circ F^{-1}(y)= \lim_{x \rightarrow +\infty} \dfrac{\left({g}/{\varphi}\right)'(x)}{F'(x)}.
\end{equation*}
Noting that $\lim\limits_{x \rightarrow +\infty} \dfrac{g(x)}{\varphi(x)}=+\infty$, $\lim\limits_{x \rightarrow +\infty} F(x)=+\infty$ and $\lim\limits_{x \rightarrow +\infty} \dfrac{g(x)}{\psi(x)}=L_{+\infty}$, by L' H\^opital's rule,
\begin{equation*}
L_{+\infty}=\lim_{x \rightarrow +\infty} \dfrac{g(x)}{\psi(x)}=\lim_{x \rightarrow +\infty}\dfrac{\left({g}/{\varphi}\right)(x)}{F(x)}=\lim_{x \rightarrow +\infty}\dfrac{\left({g}/{\varphi}\right)'(x)}{F'(x)}.
\end{equation*}
This with the previous equation concludes that  $$\lim_{y \rightarrow +\infty} G'(y)=L_{+\infty}.$$
\end{proof}


\begin{lemma}\label{lemma-T=X}
Let $T$ be the space of points on tangent lines to $G$ on  $(0,F(l)]$, i.e.,
 $$ T\triangleq \Big\{(y,v)\in \mathbb R^2: y>0, ~ v=L_z(y)=G'(z)(y-z)+G(z) \mbox{ for some } z\in (0,F(l)]\Big\},$$
 where we use the convention that $G'(F(l))$ means $G'(F(l)^-)$.  Then $T=X$, where
 $$X \triangleq \Big\{ (y,v)\in \mathbb{R}^2:~ y>0,~ v\geq \tilde{G}(y)\Big\}$$ with $\tilde{G}$ defined as
\begin{equation*}
\tilde{G}(y) \triangleq
\begin{cases}
G(y), & \mbox{if} \ 0< y < F(l) \\
G'(F(l)^-)(y-F(l))+ G(F(l)), & \mbox{if} \ y \geq F(l).
\end{cases}
\end{equation*} 
\end{lemma}

\begin{proof} It is easy to see that $T\subset X$, noting that $G$ is concave on $(0,F(l))$. Now we show $X\subset T$.  Fix $(y,v)\in X$. Define 
$$F_{y,v}(z)=v-G(z)-G'(z)(y-z), z\in (0, F(l)]. $$
First note that $G'(0^+)=+\infty,$ and hence $F_{y,v}(0^+)=-\infty.$

On the other hand, when $0<y<F(l)$, 
$$F_{y,v}(y)=v-G(y)\ge0;$$
when $y\ge F(l)$,
\begin{align*}
F_{y,v}(F(l)^-)&=v-G(F(l))-G'(F(l)^-)(y-F(l))\\
&= v-\tilde G(y)\ge 0.
\end{align*}
Therefore, by the intermediate value theorem, there exists $\omega\in (0, F(l)]$ such that $F_{y,v}(\omega)=0,$ i.e.,
\begin{equation}\label{eq-T-X}
v=G'(\omega)(y-\omega)+G(\omega),
\end{equation}
 and hence  $(y,v)\in T$. The proof is concluded. 
\end{proof}

\begin{remark}\label{remark-T=X} We point out that the tangent line that passes through a point $(y,v)\in T$  with $y>F(l)$ is unique. Indeed, assume $0<z_1<z_2\le F(l)$ satisfy \eqref{eq-T-X}, then by the strict concavity of $G$, we have
\begin{equation*}
\dfrac{G'(z_2)(y-z_2)-G'(z_1)(y-z_1)}{z_2-z_1}<\dfrac{G'(z_1)(y-z_2)-G'(z_1)(y-z_1)}{z_2-z_1}=-G'(z_1);
\end{equation*}
on the other hand, in view of equation (\ref{eq-T-X}) and the strict concavity of $G$ on $(0,y\wedge F(l)]$, one finds that
\begin{equation*}
\dfrac{G'(z_2)(y-z_2)-G'(z_1)(y-z_1)}{z_2-z_1}=\dfrac{(v-G(z_2))-(v-G(z_1))}{z_2-z_1}=\dfrac{G(z_1)-G(z_2)}{z_2-z_1}>-G'(z_1).
\end{equation*}
A contradiction occurs and the uniqueness is obtained.
\end{remark}

\begin{lemma}\label{lemma-uniqueness of non-linear equations}
The following system of equations
\begin{equation}\label{non-linear system of equations}
G'(z_1)=\dfrac{G(z_2)-G(z_1)}{z_2-z_1}=G'(z_2), ~~ 0<z_1<F(l)<z_2,
\end{equation}
 has at most one solution.
\end{lemma}

\begin{proof}
When $x_g=+\infty$, it is apparent that (\ref{non-linear system of equations}) has no solution.
When $x_g<+\infty$, we shall prove the result by contradiction. Assume that $(z_1,z_2)$ and $(\tilde z_1,\tilde z_2)$  are two distinct solutions to (\ref{non-linear system of equations}). Without loss of generality, we assume $\tilde{z}_2>z_2$.  Since $G$ is strictly concave on $(0,F(l))$ and $(F(x_g),+\infty)$, we have $G'(\tilde{z}_2)<G'(z_2)$,  and hence $\tilde{z}_1>z_1$. This implies that the line through $(\tilde z_1, G(\tilde z_1))$ and $(\tilde z_2, G(\tilde z_2))$ intersects the line through $(z_1, G(z_1))$ and $( z_1, G( z_1))$ at two distinct points, the first coordinates of which are in $(z_1, \tilde z_1)$ and $(z_2, \tilde z_2)$, respectively. Then a contradiction occurs, and the proof is concluded.
\end{proof}

We shall solve the ESO problem (\ref{avf2}) in the following three cases :\\
 (A1) \ $L_{+\infty}\leq G'(F(l)^-)$;\\
(A2) \ $L_{+\infty}> G'(F(l)^-)$ and $\lim\limits_{y\rightarrow +\infty} G(y)-L_{z_{L_{+\infty}}}(y)>0$;\\
(A3) \ $L_{+\infty}> G'(F(l)^-)$ and $\lim\limits_{y\rightarrow +\infty} G(y)-L_{z_{L_{+\infty}}}(y)\leq 0$.\\
In (A2) and (A3), $z_{L_{+\infty}}$ is the unique solution of $G'(z)=L_{+\infty}$ on $z\in (0,F(l))$ and $L_{z_{L_{+\infty}}}$ is the tangent line of $G$ at $z_{L_{+\infty}}$. 
Note that by Lemma \ref{lemma-G}, the condition $L_{+\infty}> G'(F(l)^-)$ ensures the existence and uniqueness of the solution to $G'(z)=L_{+\infty}$ on $z\in (0,F(l))$.

\begin{remark} \label{cases}
\textbf{(Well-definedness)}
To show that the limit in (A2) and (A3) is well-defined, we set $B(y)\triangleq G(y)-L_{z_{L_{+\infty}}}(y)$. Then  when $x_g<+\infty$, $B(y)$ is strictly increasing on $((F(x_g),+\infty)$ since $G$ is strictly concave on $ ((F(x_g),+\infty)$ with $\lim\limits_{y \rightarrow +\infty} G'(y)=L_{+\infty}$; when $x_g=+\infty$, $B(y)$ is non-increasing on $((F(l),+\infty)$ since $G$ is convex on $ ((F(l),+\infty)$ with $\lim\limits_{y \rightarrow +\infty} G'(y)=L_{+\infty}$. Therefore, $\lim\limits_{y\to+\infty}B(y)$ is well-defined.
\end{remark}

We split our problems into the above three cases because, as we shall see in the proof of Theorem \ref{thm3.1}, equation \eqref{non-linear system of equations} has a unique solution in cases (A1) and (A2), while no solution in (A3),  because of which the value function $V(x)$ has formula \eqref{V*(x,s)} for cases (A1) and (A2), while a different formula \eqref{2 V*(x,s)} for case (A3).


The following theorem is the first main result for the ESO pricing problem, which provides a closed-form formula for $V(x)$ and characterizing the stopping region and the optimal strategy. 


\begin{theorem}\label{thm3.1}
The option pricing problem (\ref{avf2}) admits the solutions:
\begin{enumerate}

\item[(i)] In cases (A1) and (A2), 
\begin{equation}\label{V*(x,s)}
V(x)=
\begin{cases}
s, & \mbox{if} \ \ 0< x\leq x_1,\\
c_1 \psi(x)+c_2 \varphi(x), & \mbox{if} \ \ x_1< x < x_2,\\
x-K, & \mbox{if} \ \ x\geq x_2.
\end{cases}
\end{equation}
with the constants 
\begin{equation}\label{c1,c2}
c_1 \ = \ \dfrac{g(x_2)\varphi(x_1)-g(x_1)\varphi(x_2)}{\varphi(x_1)\psi(x_2)-\varphi(x_2)\psi(x_1)} , \ \ \ \ 
c_2 \ = \  \dfrac{g(x_1)\psi(x_2)-g(x_2)\psi(x_1)}{\varphi(x_1)\psi(x_2)-\varphi(x_2)\psi(x_1)} ,
\end{equation}
where the two critical levels $x_1$ and $x_2$ (with $x_1 <l\leq x_g<x_2$) are uniquely determined by 
\begin{equation} \label{x1,x2 equation 1}
G'\left(F(x_1)\right)=\dfrac{G\left(F(x_2)\right)-G\left(F(x_1)\right)}{F(x_2)-F(x_1)}=G'\left(F(x_2)\right),
\end{equation}
or equivalently,
\begin{equation}\label{x1,x2 equation 2}
\dfrac{\left( \dfrac{g}{\varphi} \right)'(x_1)}{F'(x_1)}=\dfrac{\left( \dfrac{g}{\varphi} \right)(x_2)-\left( \dfrac{g}{\varphi} \right)(x_1)}{F(x_2)-F(x_1)}=\dfrac{\left( \dfrac{g}{\varphi} \right)'(x_2)}{F'(x_2)}.
\end{equation}
Moreover, the optimal strategy is given by
\begin{equation}\label{optimal stopping time tau star }
\tau^\star:= \inf \lbrace t\geq 0: X_t \leq x_1  \ \ \mbox{or} \ \ X_t \geq x_2 \rbrace.
\end{equation}

\item[(ii)]
In case (A3), 
\begin{equation}\label{2 V*(x,s)}
V(x)=
\begin{cases}
s, & \mbox{if} \ \ 0<x\leq x_1,\\
c_1 \psi(x)+c_2 \varphi(x), & \mbox{if} \ \ x>x_1.
\end{cases}
\end{equation}
with the constants 
\begin{equation}\label{2 c1,c2}
c_1=L_{+\infty},  \ \ \ \  c_2=\dfrac{g(x_1)}{\varphi(x_1)}-L_{+\infty}\cdot F(x_1)
\end{equation}
and the critical level $x_1$ (with $x_1 < l$) given by
\begin{equation} \label{x1 equation}
\left(\dfrac{g}{\varphi}\right)'(x_1)=L_{+\infty}\cdot F'(x_1)
\end{equation}
 In this case,  the first exit time of $X$ from the continuation region $\{x: V^*(x,s)>g(x)\}=(x_1,\infty)$ is not an optimal stopping strategy.
\end{enumerate}
\end{theorem}

\begin{proof}
We shall prove the two parts of the theorem separately.\\
{\bf Proof of (i).} In view of Lemma \ref{lemma-G},  if $x_g=+\infty$, then the convexity of $G$ on $(F(l), +\infty)$ implies that $L_{+\infty}\ge G'(F(l)^+)>G'(F(l)^-)$,
hence  $x_g<+\infty$ in case (A1); furthermore,  $G(y)$ cannot dominate $L_{z_{L_{+\infty}}}(y)$ when $y$ is large, since $G(F(l))< L_{z_{L_{+\infty}}}(F(l))$ and $G'(y)\le L_{+\infty}$ on $(F(l), +\infty)$ , and hence $x_g<+\infty$ in case (A2). We shall divide our proof into three steps.

{\bf Step 1.}  We first show that in cases (A1) and (A2), we can find  $z_0 \in (0,F(l))$
%
such that the tangent line $L_{z_0}$ intersects with $G$ at two distinct points on $(F(l),+\infty)$.

Let $z_{*}$ be the unique solution of $G'(z^-)=L_{+\infty} \vee G'(F(l)^-)$ on $z\in (0,F(l)]$. As usual, $L_{z_{*}}$ is the tangent line of $G$ at $z_{*}$. Then, by the properties of $G$ given in Lemma \ref{lemma-G},  in cases (A1) and (A2) there exists some point $(y_0,v_0)\in X$ such that  $y_0>F(l)$ and $L_{z_{*}}(y_0)<v_0<G(y_0)$. Indeed, in case (A1), note that $z_*=F(l)$ and the existence follows from the property $G'(F(l)^+)>G'(F(l)^-)$; in case (A2), it follows from that $G(y)$ dominates $L_{z_{L{+\infty}}}$ for large values of $y$.

 Since $T=X$ by Lemma \ref{lemma-T=X}, there exists $z_0 \in (0,F(l))$ such that the point $(y_0,v_0)$ is just on the tangent line $L_{z_0}$, that is, $L_{z_*}(y_0)<L_{z_0}(y_0)=v_0<G(y_0)$. According to Remark \ref{remark-T=X}, $z_0$ is unique. Also note that the strict concavity of $G$ on $(0,F(l))$ implies that for any $y_1<y_2\leq F(l)$, the line $L_{y_1}$ dominates the line $L_{y_2}$ on $[y_2,+\infty)$, and hence $z_0 \in (0,z_*)$ and $L_{z_0}(F(l))>L_{z_*}(F(l))\geq G(F(l))$. Moreover, $z_0 \in (0,z_*)$ implies that  $G'(z_0)>G'(z_*)\geq L_{+\infty}$. As a consequence, there must exist some point $n_0\in (y_0,+\infty)$ such that $L_{z_0}(n_0)>G(n_0)$ since $\lim_{y\rightarrow +\infty} G'(y)=L_{+\infty}$ by Lemma \ref{lemma-G}. Therefore,  $$(L_{z_0}-G)(F(l))>0, (L_{z_0}-G)(y_0)<0, \text{ and } (L_{z_0}-G)(n_0)>0$$ with $F(l)< y_0< n_0<+\infty$.  Hence, due to the continuity of $L_{z_0}-G$, there exist $s_0 \in (F(l),y_0)$ and $b_0 \in (y_0,n_0)$, such that $$L_{z_0}(s_0)=G(s_0)  \text{ and } L_{z_0}(b_0)=G(b_0).$$

 {\bf Step 2.} In this step,  we shall find a unique solution to \eqref{non-linear system of equations}.
 
 Note that the convex/concave property in Lemma \ref{lemma-G} implies that for any $z\in (0,F(l))$, the equation $L_z(y)=G(y)$ has at most two distinct solutions on $y\in ((F(l),+\infty)$.  Furthermore, $G(y)>L_{z_0}(y)$ for $y\in(s_0, b_0)$ and $b_0>F(x_g).$  Hence, by the property of $G$ given in Lemma \ref{lemma-G} , the set $\{ y>F(l): G(y)\geq L_{z}(y) \}$ is a closed interval, for $z\in [z_1,z_0]\subset [0,F(l))$,
 where
 
\begin{equation}\label{eq-z1}
z_1 \triangleq \inf \Big\{ z\leq z_0: L_z \text{ intersects with $G$ on } (F(l),+\infty) \Big\}.
\end{equation}

%
We denote, for $z\in[z_1, z_0]$, $$[s(z),b(z)] \triangleq \{ y>F(l): G(y)\geq L_{z}(y) \}.$$ Note that the strict concavity of $G$ on $(0,F(l))$ implies that for any $y_1<y_2\leq F(l)$, $L_{y_1}$ dominates $L_{y_2}$ on $[y_2,+\infty)$, and hence $s(z)$ is decreasing while $b(z)$ is increasing with respect to $z\in[z_1,z_0]$. Thus, $F(l)<s_0\leq s(z)\le b(z)\leq b_0< n_0$.  Also note that $b(z)\ge F(x_g),$ as the second intersection between $G$ and a line on $(F(l), +\infty)$ must occur where $G$ is concave.
 
Now we show that $s(z_1)=b(z_1)$. Assume $s(z_1)<b(z_1)$,   then for any $y\in (s(z_1),b(z_1))$, there exists $L_z$ such that $L_z(y)=G(y)$  by Lemma \ref{lemma-T=X}. Note that $L_z(y)=G(y)>L_{z_1}(y)$, and hence $z<z_1$. This is a contradiction with the definition of $z_1$.

Therefore, $L_{z_1}$ is tangent to $G$ at $z_1\in (0,F(l))$ and $z_2 \in [F(x_g),+\infty)$ where $z_1$ is given in \eqref{eq-z1} and $z_2\triangleq s(z_1)=b(z_1).$ Equivalently, the pair $(z_1,z_2)$ is a unique solution to \eqref{non-linear system of equations}, i.e.,
\begin{equation*}
 G'(z_1)=\dfrac{G(z_2)-G(z_1)}{z_2-z_1}=G'(z_2).
 \end{equation*} 
 
{\bf Step 3.} In this step, we construct the smallest concave majorant of $G$,  find an expression for the option price $V(x)$ in \eqref{avf2}, and identify the optimal stopping time. 

Define the function $L_{z_1,z_2}:[0,+\infty)\rightarrow \mathbb{R}$ as follows,
 \begin{equation*}
 L_{z_1,z_2}(y) \triangleq G(z_1)+ \dfrac{G(z_2)-G(z_1)}{z_2-z_1}(y-z_1), \ \ \ y\in [0,+\infty).
 \end{equation*}
The smallest non-negative concave majorant of $G$ on $[0,+\infty)$ is given by 
\begin{equation}\label{eq-major1}
\tilde{W}(y)=
\begin{cases}
G(y), & \mbox{if} \ y \in [0,z_1]\cup [z_2,+\infty), \\
L_{z_1,z_2}(y), & \mbox{if} \ y \in (z_1,z_2).
\end{cases}
\end{equation}
 Indeed, $\tilde{W}=G$ is obviously the smallest concave majorant of $G$ on $[0,z_1]\cup [z_2,+\infty)$, and on $(z_1, z_2)$, the line segment $L_{z_1,z_2}$ is the smallest concave curve that connects $(z_1,G(z_1))$ with $(z_2,G(z_2))$.

Denote $x \triangleq F^{-1}(y)$, $y\in [0,+\infty)$, and $x_i \triangleq F^{-1}(z_i)$, $i=1,2$. Then, $x_1 \in (0, l)$, $x_2 \in (l, +\infty)$. By \cite[Proposition 5.12]{DK03}, we have
\begin{equation}\label{eq-v-1}
V(x)=\varphi (x) \tilde{W} (F(x))=
\begin{cases}
g(x), & \mbox{if} \ x \in (0,x_1] \cup [x_2,+\infty) \\
\varphi(x) L_{z_1,z_2}(F(x)), & \mbox{if} \ x \in (x_1,x_2)
\end{cases}
\end{equation}
 Now, it is clear that  $(x_1, x_2)$ satisfies (\ref{x1,x2 equation 1}) or (\ref{x1,x2 equation 2}), since $(z_1, z_2)$ solves equation (\ref{non-linear system of equations}), and direct calculations yield
\begin{equation*}
V(x)=
\begin{cases}
s, & \mbox{if} \ x\leq x_1,\\
c_1 \psi(x)+c_2 \varphi(x), & \mbox{if} \ x_1< x < x_2,\\
x-K, & \mbox{if} \ x\geq x_2.
\end{cases}
\end{equation*}
with the constants 
\begin{equation*}
c_1 \ = \ \dfrac{g(x_2)\varphi(x_1)-g(x_1)\varphi(x_2)}{\varphi(x_1)\psi(x_2)-\varphi(x_2)\psi(x_1)} , \ \ \ \ 
c_2 \ = \  \dfrac{g(x_1)\psi(x_2)-g(x_2)\psi(x_1)}{\varphi(x_1)\psi(x_2)-\varphi(x_2)\psi(x_1)} .
\end{equation*}
The continuation region is 
\begin{equation*}
C^\star \triangleq \lbrace x \geq 0:V(x)>g(x) \rbrace=F^{-1}(\tilde{C}^\star)=F^{-1}\left( (z_1,z_2) \right)=(x_1,x_2).
\end{equation*}
where \begin{equation*}
\tilde{C}^\star \triangleq \lbrace y\in [0,+\infty): \tilde{W}(y)>G(y) \rbrace=(z_1, z_2).
\end{equation*}
Therefore, the stopping region  $\Gamma^\star \triangleq \lbrace x\geq 0: V(x)=g(x) \rbrace=(x_1, x_2)^C$ is the complement of $C^\star$, and consequently the optimal stopping time is
\begin{equation*}
\tau^\star = \inf \Big\{ t\geq 0: X_t \in  \Gamma^\star \Big\} = \inf \Big\{ t\geq 0: X_t \leq x_1 \ \ \mbox{or} \ \ X_t \geq x_2 \Big\}.
\end{equation*}

\noindent{\bf Proof of (ii).} In case (A3), note that for any $z\in(z_{L_{+\infty}},F(l)]$, there exists some $y_z\in [F(l),+\infty)$ such that $G(y)>L_z(y)$ for all $y>y_z$; for any $z\in (0,z_{L_{+\infty}}]$, we have $L_z(y)>G(y)$ on $y\in (z,+\infty)$. As a consequence, by the properties of $G$ given in Lemma \ref{lemma-G}, one cannot find a line tangent to $G$ at $(t_1, G(t_1)$ and $(t_2, G(t_2)$ such that $t_1\in (0, F(l))$ and $t_2\in (F(l), +\infty)$, and hence there is no solution for (\ref{non-linear system of equations}) in this case. By setting $z_1=z_{L_{+\infty}}$, the tangent line of $G$ at $z_1$ is then given as
\begin{equation*}
 L_{z_1}(y) \triangleq G(z_1)+ L_{+\infty}(y-z_1), \ \ \ y\in [0,+\infty)
 \end{equation*}
The smallest non-negative concave majorant of $G$ on $(0,+\infty)$ is given by 
\begin{equation*}
\tilde{W}(y)=
\begin{cases}
G(y), & \mbox{if} \ y \in [0,z_1), \\
L_{z_1}(y), & \mbox{if} \ y \in [z_1,+\infty).
\end{cases}
\end{equation*}
Indeed, suppose there is a concave majorant $\bar W$ of  $G$ smaller than $\tilde W$, i.e, $\bar W\neq \tilde W$ and $\bar W(y)\le \tilde W(y)$ for all $y\ge 0$. Note that $\bar W$ must coincide with $\tilde W$ on $[0,z_1].$ Then, there must exist a number $\eta \in (z_1,+\infty)$ such that $G(\eta)\le \bar W(\eta) < \tilde W(\eta).$   As $\bar W$ is concave, we have $\bar W'(\eta^+)\le \bar W'(\eta^-)<\bar W'(z_1^+)\le \bar W'(z_1^-)=L_{+\infty}.$  Therefore,  on  $ (\eta,+\infty)$, $\bar W$ is dominated by the line $\tilde W(\eta)+\bar W'(\eta_+)(y-\eta),$ which is below the graph of $G$ for large values of $y$, since $\bar W'(\eta_+)<L_{+\infty}$. A contradiction occurs, and hence $\tilde W$ is the smallest concave majorant.

As a consequence, we have
\begin{equation*}
V(x)=
\begin{cases}
s, & \mbox{if} \ \ 0<x\leq x_1,\\
c_1 \psi(x)+c_2 \varphi(x), & \mbox{if} \ \ x>x_1.
\end{cases}
\end{equation*}
where constants $c_1$ and $c_2$ given by (\ref{2 c1,c2}), and the critical level $x_1$ is uniquely determined by (\ref{x1 equation}) and $x_1<l$. In this case, the first exit time of $X$ from the continuation region is not an optimal strategy,  by  \cite[Proposition 5.14]{DK03},  since $+\infty>L_{+\infty}>G'(F(l)^-)>0$ and $(l,+\infty)\subset(x_1,\infty)$.
\end{proof}

\begin{remark}
Let $\Gamma$  denote the stopping region, i.e., 
$\Gamma \triangleq \{ x \in \mathcal{I}: V(x)=g(x) \}.$
Then in cases (A1) and (A2), $\Gamma=(0,x_1]\cup [x_2,+\infty)$ is two-sided, while in case (A3), $\Gamma=(0,x_1]$ is one-sided.
\end{remark}

In the following Propositions \ref{smooth-fit}  and  \ref{prop-V}, we provide some regularity results which are well known in Black-Scholes Model. We shall only prove the results for cases (A1) and (A2), and case (A3) can be treated in a similar way without extra difficulty.

%

\begin{proposition}[Smooth-fit principle] \label{smooth-fit}
 The value function $V(x)$ satisfies the smooth-fit condition $V'(x_i)=g'(x_i)$ with $x_i$ being the critical levels (optimal stopping points)  given in Theorem \ref{thm3.1}. Moreover, the value function $V$ is continuously differentiable on $\mathbb R^+$,  and  twice continuously differentiable on $\mathbb R^+\backslash\{x_1, x_2\}$. 
\end{proposition}

\begin{proof}  {First note that $V(x)=\varphi (x) \tilde{W} (F(x))$, and it is differentiable at $x_1$ and $x_2$, since $\tilde W$ given by \eqref{eq-major1} is differentiable at $F(x_1)$ and $F(x_2)$. By \eqref{eq-v-1}, it is easy to check that $V$ is continuously differentiable on $\mathbb R^+$, and $V'(x_i)=g'(x_i), i=1,2.$}


By \eqref{eq-v-1}, $V$ is twice differentiable on $\mathbb R^+\backslash\{x_1, x_2\}$. Finally,  using $V(x)=\varphi (x) \tilde{W} (F(x))$, one obtains 
\begin{equation}\label{sign of the difference of W''}
V''(x_i^+) - V''(x_i^-) = \varphi(x_i) (F'(x_i))^2\cdot \left(\tilde{W}''(z_i^+)-\tilde{W}''(z_i^-)\right)
\end{equation}
with $z_i=F(x_i)$. In view of the linearity of $\tilde{W}$ on the transformed continuation region $(z_1,z_2)$ and the strict concavity of $\tilde{W}$ on the transformed stopping region $(0,z_1)\cup (z_2,+\infty)$, one finds that 
$\tilde{W}''(z_1^+)-\tilde{W}''(z_1^-)=0-\tilde{W}''(z_1^-)>0$ and $\tilde{W}''(z_2^+)-\tilde{W}''(z_2^-)=\tilde{W}''(z_2^+)-0<0$. 
Therefore,
\begin{equation} \label{non twice differentiable}
 V''(x_1^+)>V''(x_1^-) \ \mbox{and} \ V''(x_2^+)<V''(x_2^-).
\end{equation}

\end{proof}


\begin{remark}[Relationship with  free-boundary problems]\label{fb}
By the smooth-fit principle, one can check that $V(x)$ is a solution to the following free-boundary problem (also known as the differential variational inequality):
\begin{equation}\label{eq-fb}
\begin{cases}
\min \{(r-\mathcal{L})V(x), V(x)-g(x)\}=0\\
V'|_{\partial C}=g'|_{\partial C} \ \ \ \  \mbox{(smooth fit)} \\
V|_{\partial C}=g|_{\partial C} \ \ \ \ \mbox{(continuous fit)}
\end{cases}
\end{equation}
where $C$ is the waiting (continuation) region in which it is optimal to wait until the asset price reaches its boundary $\partial C$.


The free-boundary problems can be  solved using a ``guess-and-verify" technique (see, e.g., \cite{Peskir06}, \cite{Pham09}). More specifically, given one specific free-boundary problem,  firstly, one needs to guess the structure of the  stopping/continuation region and impose proper conditions (such as the smooth-fit and continuous-fit conditions) on the boundary of the stopping region; secondly, based on the guess and assumptions made in the first step, one may solve the free-boundary problem  through standard techniques from the theory of ODE/PDE; the final step is to verify that the solution obtained from the second step does satisfy the conditions imposed in the first step. In contrast, for  a general free-boundary problem \eqref{eq-fb}, Theorem \ref{thm3.1}  directly identifies the stopping/continuation regions  in cases (A1), (A2) and (A3),  provides a complete characterization of stopping strategies,  and gives  closed-form formulae for the solution. 
\end{remark}

\begin{proposition}\label{prop-V} The value function $V(x)$ is   increasing on $\mathcal I$.  If  $\varphi$ and $\psi$ are convex on $\mathcal I$, then $V(x)$ is globally convex, and in particular, it is  strictly convex and strictly increasing on the continuation region $C\triangleq\mathcal I\backslash \Gamma$. 
\end{proposition}

\begin{proof}
By \eqref{avf2} and the comparison principle for SDE \eqref{eq-X}, $V(x)$ is globally increasing on $\mathcal I$.  Now we prove the result when $\varphi$ and $\psi$ are convex on $\mathcal I$ . By Theorem \ref{thm3.1}, $V(x)$ is increasing and convex on the stopping regions, and by Proposition \ref{smooth-fit}, $V(x)$ is continuously differentiable on $\mathbb R^+$. Thus,  to prove the result, it suffices to prove that $V(x)$ is strictly convex and strictly increasing on the continuation region. 

First note that when $\psi(x)$ and $\varphi(x)$ are convex functions,  at least one of $\psi''(\cdot)$ and $\varphi''(\cdot)$ is non-zero on the continuation region. Indeed, since $\psi$ and $\varphi$ are two fundamental solutions of the ODE: $(\mathcal{L}-r)u(x)=0$, we have $\psi'\varphi'' -\varphi'\psi'' = \dfrac{2r}{\sigma^2(x)} (\varphi\psi' - \psi\varphi')>0$. Thus, to obtain the strict convexity of $V$ on $C$, it suffices to show  $c_1>0$ and $c_2>0$. 

Clearly $c_1$ given in \eqref{2 c1,c2} is positive. In (\ref{c1,c2}),  $c_1>0$ since $g(x)/\varphi(x)$ and $\psi(x)/\varphi(x)$ are strictly increasing. Note that the smooth-fit principle guarantees $V$ is differentiable at $x_1$, which yields
\begin{equation*}
c_1\psi'(x_1)+c_2\varphi'(x_1)=V'(x_1^+)=V'(x_1^-)=g'(x_1)=0,
\end{equation*}
and consequently $c_2>0$ since $c_1>0$ and $\psi'(x_1)>0, \varphi'(x_1)<0$.

Finally, the strict convexity of  $V$ on $C$ implies that $V'(x)> V'(x_1)=g'(x_1)=0$ for all $x \in C$, which implies that $V$ is strictly increasing on $C$.
\end{proof}

In reality, it is desirable to know the impact of the external parameters, such as the interest rate $r$, the minimum guarantee $l$ and the strike price $K$ (which can also be interpreted as the transaction cost), on investors' exercise strategy. 

\begin{proposition}[External comparison principles]\label{external comparison principles under ESO problem}
The critical level $x_1 \ (x_2)$ is increasing (decreasing) with respect to $r$; $x_1 \ (x_2)$ is decreasing (increasing) with respect to $K$; both $x_1$ and $x_2$ are increasing with respect to $l$.
\end{proposition}

\begin{proof}
For the sake of convenience, let us denote the value function $V^{\star}(x,s)$ in (\ref{avf2}) by $V_l(x)$ to emphasize its dependence on $l$. Similar notations are also used for parameters $r$ and  $K$. Noting that $V(x)$ is decreasing in $r$ and $K$, and increasing in $l$, we have
\begin{equation*}
V_{r_1}(x)\ge V_{r_2}(x) \ \mbox{for} \ r_1<r_2 ; \ \ V_{K_1}(x)\ge V_{K_2}(x) \ \mbox{ for} \ K_1<K_2; \ \ V_{l_1}(x)\ge V_{l_2}(x) \ \mbox{ for} \ l_1<l_2.
\end{equation*}
We denote by $C_r$, $C_K$ and $C_l$, the continuation regions corresponding to $V_{r}(x)$, $V_{K}(x)$ and $V_{l}(x)$. Then, $V_{r_1}(x)\ge V_{r_2}(x)$  implies
\begin{equation} \label{C_rr}
C_{r_2}=\lbrace x: V_{r_2}(x)>g(x) \rbrace \subset \lbrace x: V_{r_1}(x)>g(x) \rbrace = C_{r_1}.
\end{equation}
Note also that 
\begin{align*}
&V_{K_1}(x)-V_{K_2}(x)\\
=&\sup_{\tau \geq 0} E_x\left[e^{-r\tau}\Big(\left((X_\tau \vee l) -K_2\right)+(K_2-K_1)\Big)\right]-\sup_{\tau \geq 0} E_x\left[e^{-r\tau}\left((X_\tau\vee l) -K_2\right)\right]\\
\leq& K_2-K_1=g_{K_1}(x)-g_{K_2}(x),
\end{align*}
i.e., $V_{K_1}(x)-g_{K_1}(x)\leq V_{K_2}(x)-g_{K_2}(x)$. Consequently, 
\begin{equation}\label{C_KK}
C_{K_1}=\lbrace x: V_{K_1}(x)>g_{K_1}(x) \rbrace \subset \lbrace x: V_{K_2}(x)>g_{K_2}(x) \rbrace = C_{K_2}.
\end{equation}
Therefore, the increment of the value $r$ $(K)$ accelerates (decelerates) rational exercises by shrinking (expanding) the continuation region.

For the parameter $l$, we denote $C_l :=(x_1^l,x_2^l)$. It is thus sufficient to show that $x_1^{l_1}\leq x_1^{l_2}$ and $x_2^{l_1}\leq x_2^{l_2}$ for $l_1<l_2$.
Suppose this is not the case so that $x_1^{l_1}> x_1^{l_2}$, and choose arbitrary $x \in (x_1^{l_2}, x_1^{l_1}\wedge x_2^{l_2})$. {By Theorem \ref{thm3.1} and Proposition \ref{prop-V},}
\begin{equation*}
V_{l_2}(x)> l_2-K \ \ \mbox{and} \ \ V_{l_1}(x)=l_1-K
\end{equation*}
for all $x \in (x_1^{l_2}, x_1^{l_1}\wedge x_2^{l_2})$.  As a result, we derive that $V_{l_2}(x)-V_{l_1}(x)>l_2-l_1$, which is impossible since $V_{l_2}(x)-V_{l_1}(x)\le l_2-l_1 $ due to the inequality $(x \vee l_2)\leq (x \vee l_1)+(l_2-l_1)$. 

Now we show that $x_2^{l_1}\leq x_2^{l_2}$. Indeed, assuming on the contrary that $x_2^{l_1}>x_2^{l_2}$ and taking any $x \in (x_2^{l_2},x_2^{l_1})$, then {$V_{l_2}(x)=x-K$ by Theorem \ref{thm3.1}, and $V_{l_1}(x)>x-K$ since $x$ is  in the continuation region of $V_{l_1}$. As a consequence, we derive that $V_{l_1}(x)>V_{l_2}(x)$ for $l_1<l_2$, which is  impossible by \eqref{avf1}}. Therefore, we conclude that $x_1^{l_1}\leq x_1^{l_2}$ and $x_2^{l_1}\leq x_2^{l_2}$.
\end{proof}

The following proposition indicates the dependence of the value function and the optimal exercise boundaries on the internal parameters (drift and volatility).

\begin{proposition}[Internal comparison principles]\label{internal comparison principles under ESO problem}
The value function $V(x)$ is non-decreasing with respect to the drift  $\mu$. That is,  $\ V_{\mu_1}(x)\leq V_{\mu_2}(x))$ for all $x\in\mathcal I$, if $\mu_1(x) \leq \mu_2(x)$ for all $x\in \mathcal I$. As a consequence, the critical level $x_1 \ (x_2)$ is non-increasing (non-decreasing) with respect to  $\mu$.

If we further assume that $\varphi$ and $\psi$ are convex on $\mathcal I$, we have a similar result for the volatility $\sigma.$ That is, the value function $V(x)$ is also non-decreasing with respect to the volatility  $\sigma$. As a consequence, the critical level $x_1 \ (x_2)$ is non-increasing (non-decreasing) with respect to $\sigma$. 
\end{proposition}

\begin{proof} First we show the comparison principle for the volatility $\sigma.$ Denote by $\Gamma_\sigma \ (\mbox{resp.} \ C_\sigma)$ the stopping region (resp. continuation region) of $V_\sigma$. Note that on $\Gamma_{\sigma_1}$, the option value $V_{\sigma_1}(x)=g(x)$ is from an immediate exercise, and hence $V_{\sigma_2}(x)\geq V_{\sigma_1}(x)$ on $\Gamma_{\sigma_1}$. Thus it suffices  to show $V_{\sigma_2}(x)\geq V_{\sigma_1}(x)$ on $C_{\sigma_1}$. 

Denote
\begin{equation*}
\mathcal{A}_{\sigma} \triangleq \mathcal{L}_{\sigma} -r=\frac{1}{2}\sigma^2(x)\frac{d^2}{dx^2} +\mu (x) \frac{d}{dx}-r,
\end{equation*}
which is the generator of the underlying diffusion $X^\sigma$ killed at a constant rate $r$. Unless otherwise specified, we will use similar simplified representations, for instance, we denote by $(\psi_{\sigma},\varphi_{\sigma})$ a pair of fundamental solutions to $(\mathcal{A}_{\sigma}u)(x)=0$. Recall that $\psi_{\sigma}(x)$ is strictly increasing whilst $\varphi_{\sigma}(x)$ is strictly decreasing. 

Next, let $\tau_{y,z}^{\sigma_1}:= \tau_y^{\sigma_1} \wedge \tau_z^{\sigma_1}$ with $\tau_y^{\sigma_1}:= \inf \lbrace t\geq 0: X_t^{\sigma_1}=y \rbrace$ and $\tau_z^{\sigma_1}:= \inf \lbrace t\geq 0: X_t^{\sigma_1}=z \rbrace$. Note that $E_x[\tau_{y,z}^{\sigma_1}]<+\infty$ (see, e.g.,  \cite[Lemma 46.1]{DW1987}).  By Dynkin's formula, for all $y \leq x \leq z$, 
\begin{align*}
E_x\left[ e^{-r\tau_{y,z}^{\sigma_1}} V_{\sigma_2}( X_{\tau_{y,z}^{\sigma_1}}^{\sigma_1}) \right] \ &= \ V_{\sigma_2}(x)+ E_x {\int}_0^{\tau_{y,z}^{\sigma_1}}e^{-rt}\left(\mathcal{A}_{\sigma_1}V_{\sigma_2}\right)\left(X_t^{\sigma_1}\right)dt  \\
\ &\leq \ V_{\sigma_2}(x)+ E_x {\int}_0^{\tau_{y,z}^{\sigma_1}}e^{-rt}\left(\left(\mathcal{A}_{\sigma_1}-\mathcal{A}_{\sigma_2}\right)V_{\sigma_2}\right)\left(X_t^{\sigma_1}\right)dt \\ 
\ &= \ V_{\sigma_2}(x)+ E_x {\int}_0^{\tau_{y,z}^{\sigma_1}}e^{-rt}\left(\dfrac{1}{2} \left(\sigma_1^2\left(X_t^{\sigma_1}\right)-\sigma_2^2\left(X_t^{\sigma_1}\right)\right)V''_{\sigma_2}\left(X_t^{\sigma_1}\right)\right)dt 
\\
\ &\leq \ V_{\sigma_2}(x)
\end{align*}
since $(\mathcal{A}_{\sigma_2}V_{\sigma_2})(x) \leq 0$ by Remark \ref{fb} and $V''_{\sigma_2}(x)\geq 0$ by Proposition \ref{prop-V}. As a consequence, we see that for $x \in (y,z) \subset C_{\sigma_1}$, 
\begin{align*}
V_{\sigma_2}(x) \ &\geq \ E_x\left[ e^{-r\tau_{y,z}^{\sigma_1}} V_{\sigma_2}( X_{\tau_{y,z}^{\sigma_1}}^{\sigma_1}) \right]  \\
&= \ V_{\sigma_2}(y)E_x\left[ e^{-r\tau_y^{\sigma_1}}; \tau_y^{\sigma_1} < \tau_z^{\sigma_1} \right] + V_{\sigma_2}(z)E_x\left[ e^{-r\tau_z^{\sigma_1}}; \tau_y^{\sigma_1} > \tau_z^{\sigma_1} \right] 
\end{align*}
On the other hand, $X_t^{\sigma_1}\in (y,z)\subset C_{\sigma_1}$, $P_x$-almost surely for $t\in [0, \tau_{y,z}^{\sigma_1})$, hence $e^{-r\left(t\wedge \tau_{y,z}^{\sigma_1}\right)} V_{\sigma_1}\left( X_{t\wedge \tau_{y,z}^{\sigma_1}}^{\sigma_1}\right) $ is a bounded continuous $P_x$-martingale, and thus
\begin{align*}
V_{\sigma_1}(x) \ =\ V_{\sigma_1}(y)E_x\left[ e^{-r\tau_y^{\sigma_1}}; \tau_y^{\sigma_1} < \tau_z^{\sigma_1} \right] + V_{\sigma_1}(z)E_x\left[ e^{-r\tau_z^{\sigma_1}}; \tau_y^{\sigma_1} > \tau_z^{\sigma_1} \right] 
\end{align*}
Substracting $V_{\sigma_1}(x)$ from $V_{\sigma_2}(x)$ and taking $y =x_1^{\sigma_1}$ and $z = x_2^{\sigma_1}$, we see that for all $x\in C_{\sigma_1}$,
\begin{align*}
\left(V_{\sigma_2}-V_{\sigma_1}\right)(x) \ &\geq \ \left(V_{\sigma_2}-V_{\sigma_1}\right)(x_1^{\sigma_1}) \ E_x\left[ e^{-r\tau_{x_1^{\sigma_1}}^{\sigma_1}}; \tau_{x_1^{\sigma_1}}^{\sigma_1} < \tau_{x_2^{\sigma_1}}^{\sigma_1} \right] \\
&\ \ \ \ + \left(V_{\sigma_2}-V_{\sigma_1}\right)(x_2^{\sigma_1}) \ E_x\left[ e^{-r\tau_{x_2^{\sigma_1}}^{\sigma_1}}; \tau_{x_1^{\sigma_1}}^{\sigma_1} > \tau_{x_2^{\sigma_1}}^{\sigma_1} \right] \\
& \geq \ 0
\end{align*} 
where the last inequality holds because $\left(V_{\sigma_2}-V_{\sigma_1}\right)(x_i^{\sigma_1})=\left(V_{\sigma_2}-g\right)(x_i^{\sigma_1})\geq 0$ for $i=1,2$.
Consequently, we obtain that $V_{\sigma_1}(x)\leq V_{\sigma_2}(x)$ for $x \in C_{\sigma_1}$. 

In a similar manner, one can show that $V_{\mu_1}(x)\leq V_{\mu_2}(x)$ for $x \in C_{\mu_1}$. \

Finally, $V_{\mu_1}\leq V_{\mu_2}$ on $C_{\mu_1}$ and $V_{\sigma_1}\leq V_{\sigma_2}$ on $C_{\sigma_1}$ imply that $C_{\mu_1}\subset C_{\mu_2}$ and $C_{\sigma_1}\subset C_{\sigma_2}$, respectively. This is because, for instance, for $x\in C_{\sigma_1}$ $g(x)<V_{\sigma_1}(x)\le V_{\sigma_2}(x)$, and hence $x\in C_{\sigma_2}.$ Therefore, the critical level $x_1$ is non-increasing with respect to $\mu$ and $\sigma$, while $x_2$ is non-decreasing with respect to $\mu$ and $\sigma$.
\end{proof}

\begin{remark}
 Proposition \ref{internal comparison principles under ESO problem} indicates that increased volatility $\sigma$ and/or drift $\mu$  results in increased option value  and expanded continuation region (longer waiting time before a rational exercise).
\end{remark}

\begin{remark} \label{remark3.6}
To prove the comparison principle with respect to the drift $\mu,$ assuming further suitable conditions on $\mu$ and $\sigma$, one may directly apply the comparison principle (see, e.g., \cite[Proposition 2.18]{KS}) for  SDE \eqref{eq-X} and then use  \eqref{avf2}. However, the properties of the value function in Proposition \ref{prop-V} is critical to obtain the comparison principle with respect to $\sigma$, since we do not have such a comparison principle for SDE with respect to the volatility.
\end{remark}

We finish this section by considering an example, the  employee stock option driven by a geometric Brownian motion,  which was studied in \cite{XG01}.
 
\begin{example}
Let $X$ be a geometric Brownian motion with drift $\mu$ and volatility $\sigma$, that is, $$X_t=X_0 \exp\left(\mu t+\sigma B_t-\frac{1}{2}\sigma^2 t\right).$$ The equation \eqref{eq-L} now is 
\begin{equation}\label{eso1}
\frac{1}{2}\sigma^2 x^2 u''(x) +\mu x u'(x) -ru(x)=0.
\end{equation}  
It is known that $\psi(x)=x^{\gamma_0}$ and $\varphi(x)=x^{\gamma_1}$ are two linearly independent solutions to \eqref{eso1}, where $\gamma_1<0<\gamma_0$ are the two solutions to 
$\frac{1}{2}\sigma^2 \gamma^2+(\mu-\frac{1}{2}\sigma^2)\gamma-r=0$.

When $\mu>r$, we have $V(x)=+\infty$. Indeed, 
\begin{align*}
V(x)\ &= \ \sup_{\tau}E_x\left[e^{-r\tau}\left((X_{\tau}\vee l)-K\right)\right]\\
&\geq \ \sup_{\tau}E_x\left[e^{-r\tau}X_{\tau}\right]-K\\
\ &\geq \ E_x[X_0 \exp\{(\mu-r)t+(\sigma B_t-\frac{1}{2}\sigma^2t)\}]-K\\
\ &= \ x \cdot e^{(\mu-r)t}-K 
\end{align*}
where the last term tends to $\infty$ as $t\to \infty.$

When $\mu<r$, we have $\gamma_0>1$, and hence $$L_{+\infty}  =\lim_{x\rightarrow +\infty} \dfrac{1}{\psi'(x)}=\lim_{x\rightarrow +\infty} \dfrac{1}{\gamma_0 x^{\gamma_0-1}}=0.$$ Therefore $L_{+\infty}<G'(F(l)^-)$ by Lemma \ref{lemma-G},  and part (i) of Theorem \ref{thm3.1} is applicable. By doing some calculations for  equations (\ref{x1,x2 equation 2}) and (\ref{c1,c2}), one may get the same value function as that in \cite[Theorem 2] {XG01}.

When $\mu=r$,  $(\mathcal{L}-r)g(x)=rK>0$ for $x>l$. It follows from Proposition \ref{Prop-H} and part (ii) of Lemma \ref{lemma-G} that $x_g=+\infty$, where $x_g$ is defined in Lemma \ref{lemma-G}, and this only happens in case (A3) (see the beginning of the proof of Theorem \ref{thm3.1}). Therefore, one may apply part (ii) of Theorem \ref{thm3.1} and get the same value function as in  \cite[Corollary 1]{XG01}.


\end{example}

\section{On an American put option with barrier}
In this section, we consider the pricing problem for an American put option with a barrier, which is formulated as the following optimal stopping problem:
\begin{equation}\label{vf}
V(x)\triangleq \sup_{\tau \geq 0} E_x[e^{-r\tau} (q-X_{\tau})^+I_{\lbrace \tau<\tau_d \rbrace }], ~ x\in (0,d) 
\end{equation}
with $\tau_d \triangleq \inf \lbrace t\geq 0: X_t = d \rbrace$, where  $d>0$ is the barrier, $q \in (0,d)$ is the strike price, and $X$ is the price process given in \eqref{eq-X}.

Denote by $\tilde{X}$ the stopped price process of $X$, which starts in $(0,d)$ and is absorbed when it reaches the barrier $d$. Then $\tilde X$ has a natural left boundary $0$ and an absorbing right boundary $d$. Now the value function $V(x)$ defined by \eqref{vf} can be written as \begin{align}\label{V(x)'}
V(x)=\sup_{\tau \geq 0} E_x[e^{-r\tau} h(\tilde{X}_\tau)], \ \ \ x\in (0,d)
\end{align}
where $h(x)\triangleq (q-x)^+$ is the reward function.

In this section, we shall apply \cite[Proposition 5.5]{DK03} to obtain a closed-form expression for  $V(x)$ given by \eqref{V(x)'} (see Theorem \ref{value function for barrier put option with American feature}), which extends the result in \cite[Section 6.1]{DK03}. As a byproduct, we also obtain some results for $V(x)$ on its regularity and comparison principles.

Throughout this section, besides {\bf Assumption A}, we also assume that $\varphi(x)$ is convex on $(0,+\infty)$. Note that a sufficient condition for the convexity of $\varphi(x)$ (and $\psi(x)$)  is the tranversality condition  \eqref{condition-tran}.

 \begin{lemma}\label{lem-mu}
Under the assumptions  that $0$ is a natural boundary and that there exists some $r_0>0$ such that the function $\theta_{r_0}(x)=r_0x-\mu(x)$ is non-decreasing, we have  $\liminf\limits_{x\to 0^+} \mu(x)\ge 0$.  \end{lemma}

\begin{proof}
Let's first establish the inequality: $\liminf\limits_{x\to 0^+} \mu(x)> -r_0 q$, for any $q>0$. Assuming on the contrary that $\liminf\limits_{x\to 0^+} \mu(x)\leq -r_0 q$ for some $q>0$. In particular, we introduce an auxiliary function $h(x)=(q-x)^+$, it then follows that $$(\mathcal{L}-r_0)h(x)=\theta_{r_0}(x)-r_0 q\geq \limsup_{x\to 0^+} \theta_{r_0}(x)-r_0 q=-\liminf_{x\to 0^+} \mu(x)-r_0 q\geq 0$$ for all $x\in(0,q)$. This implies that $H(y)$ given in \eqref{H} is convex on $(0,F(q))$ since $(\mathcal{L}-r_0)h(x)$ preserves the sign of $H''$ by Proposition \ref{Prop-H}. On the other hand, 
\begin{align*}
H'(F(q)^-) \ &\triangleq \  \lim_{y\rightarrow F(q)^-} H'(y)\\
&= \ \lim_{x\rightarrow q^-} \left(\dfrac{h}{\varphi}\right)'(x)/F'(x)\\
&= \ \lim_{x\rightarrow q^-} \dfrac{-\varphi-(q-x)\varphi'}{\psi'\varphi-\psi\varphi'}(x)\\
&= \ \dfrac{-\varphi}{\psi'\varphi-\psi\varphi'}(q)<0.
\end{align*}
 Therefore, $H'(x)<0$ for  $x\in (0,F(q))$, and hence $H(0+)>H(F(q))=0$. Noting that $0$ is a natural boundary or an entrance-not-exit point implies that $\varphi(0)=+\infty$, which ensures $H(0+)=0$. A contradiction occurs. Therefore, we have $\liminf\limits_{x\to 0^+} \mu(x)>-r_0 q$, for any $q>0$, and consequently $\liminf\limits_{x\to 0^+} \mu(x)\geq 0$. 
\end{proof}

In the following lemma, the concavity/convexity of $H(\cdot)$ is described based on the sign of $\mu(q^-)$. Note that since $\theta(x)=rx-\mu(x)$ is non-decreasing, the left limit $\mu(q^-)$ is well-defined.

\begin{lemma}\label{second-order prop-H}
The function $H(\cdot)$ given by \eqref{H} belongs to $C^2(0,F(q))$ and it possesses the following  properties:
\begin{enumerate}
\item[(i)] If  $\mu(q^-)\geq 0$, $H(\cdot)$ is concave on $(0,F(q))$.
\item[(ii)] If $\mu(q^-)<0$, then there exists a unique point $x_{\theta} \in (0,q)$ such that $H(\cdot)$ is concave on $(0,F(x_\theta))$ and strictly convex on $(F(x_\theta), F(q))$.
\end{enumerate}
\end{lemma} 

\begin{proof}
Note that for $x\in (0,q)$, $$(\mathcal{L}-r)h(x)=\frac{\sigma^2(x)}{2}h''(x)+\mu(x)h'(x)-rh(x)=\theta(x)-rq.$$ 

In case $(i)$, since $\theta(\cdot)$ is non-decreasing, it follows that $(\mathcal{L}-r)h(x)\leq \theta(q^-)-rq=-\mu(q^-)\leq 0$ for all $x\in (0,q)$. Hence $H(\cdot)$ is concave on $(0, F(q))$ by Proposition \ref{Prop-H}.

In case $(ii)$, since $\theta(q^-)-rq=-\mu(q^-)>0$ and $\limsup\limits_{x\to 0^+} \theta(x)-rq=-\liminf\limits_{x\to 0^+} \mu(x)-rq\le-rq<0$ (by lemma \ref{lem-mu}), noting that $\theta(\cdot)$ is non-decreasing, there must exist a unique point $x_\theta\in(0,q)$ such that $\theta(x)-rq\leq 0$ on $(0,x_\theta)$ while $\theta(x)-rq> 0$ on $(x_\theta,q)$. Indeed, one may  take $x_\theta=\sup \{ a\in (0,q):\theta(x)-rq\le 0  \text{ on } (0,a) \}$.
\end{proof}


Furthermore, $H(\cdot)$ in \eqref{H} also possesses the following property.

 \begin{lemma}\label{monotone prop-H}
There exists a unique $y_h \in (0,F(q))$ such that $H(\cdot)$ is strictly increasing on $(0,y_h)$ and strictly decreasing on $(y_h, F(q))$ with $H'(y_h)=0$.
\end{lemma} 

\begin{proof}
In view of \eqref{W} and \eqref{H'},  $H'(x)$  in \eqref{H} is a positive multiple of the function $Q(x) \triangleq \varphi(x)h'(x)-h(x)\varphi'(x)$. Note that $Q$ is continuous and strictly decreasing on $(0,q)$ since $Q'(x)=-h(x)\varphi''(x)<0$. Noting that $H(\cdot)$ is continuous and strictly positive on $(0,F(q))$ with $H(0)=H(F(q))=0$, we have $Q(x)>0$ nearby $0$  and $Q(x)<0$ nearby $F(q)$. Therefore, the equation $Q(x)=0$ has a unique solution denoted by $y_h$ in  $(0, F(q))$. Thus, $H(\cdot)$ is strictly increasing on $(0,y_h)$ and strictly decreasing on $(y_h, F(q))$ with $H'(y_h)=0$.
\end{proof}

Define $$y_H \triangleq \inf \{ y \in (0,F(d)): H''(y)\geq 0 \}.$$ The $H''(y)\geq 0$ for all $y\in (y_H,F(q))$ by Lemma \ref{second-order prop-H}. Furthermore, we  have $y_h<y_H$. Indeed, if $y_H\leq y_h$, by Lemma \ref{monotone prop-H} , $H'(y_H)\ge 0$ and hence $H'(y)\ge 0$ for all $y\in (y_H,F(q))\supset (y_h,F(q))$, which contradicts with Lemma \ref{monotone prop-H}.

 On the other hand, by the construction of $y_H$ and $y_\theta \triangleq F(x_\theta)$, it is apparent that $y_H\leq y_\theta\le F(q)$. Note that $H''(y)=0$ for $y\in (y_H, y_\theta)$, and hence the graph of $H(\cdot)$ on $y\in (y_H, y_\theta)$ is a line segment.

\begin{lemma}\label{lemma-TH=XH}
Define $$ T_H \triangleq \Big\{(y,v)\in \mathbb R^2: y>0, ~ v= L_z(y)\triangleq H'(z)(y-z)+H(z) \mbox{ for some } z\in [y_h,y_H]\Big\},$$
 where we use the convention that $H'(y_H)$ means $H'({y_H}^-)$.  Then $T_H=X_H$, where
 $$X_H \triangleq \Big\{ (y,v)\in \mathbb{R}^2:~ y>0,~ v\geq \tilde{H}(y)\Big\}$$ with $\tilde{H}$ defined as
\begin{equation*}
\tilde{H}(y) \triangleq
\begin{cases}
H(y_h), & \mbox{if} \ 0< y \leq y_h \\
H(y), & \mbox{if} \ y_h< y < y_H \\
H'({y_H}^-)(y-y_H)+ H(y_H), & \mbox{if} \ y \geq y_H.
\end{cases}
\end{equation*} 
\end{lemma}

\begin{proof}
The proof is similar to that of  Lemma \ref{lemma-T=X}.
\end{proof}

\begin{lemma} \label{uniqueness z_0}
The following equation has a unique solution $z_0 \in [y_h,y_H]$:
\begin{equation}\label{equation for unique z_0}
H'(z_0)=\dfrac{H(F(d))-H(z_0)}{F(d)-z_0} \equiv \dfrac{-H(z_0)}{F(d)-z_0}\, .
\end{equation}
\end{lemma}

\begin{proof} First we show that the point $(F(d),0)$ is above the line  $L_{y_H}(y) \triangleq H'({y_H}^-)(y-y_H)+ H(y_H)$.  Note that $H'(y_H^-)< 0$ by Lemma \ref{monotone prop-H} and the fact $y_h<y_H$.

For case $(i)$ in Lemma \ref{second-order prop-H}, if $y_H=F(q)$, $L_{y_H}(F(d))=H'(y_H^-)(F(d)-F(q))< 0$; if $y_H<F(q)$, the concavity of $H(\cdot)$ on $(0,F(q))$ now implies that $H''(y)=0$ for all $y \in (y_H,F(q))$. That is, $H(\cdot)$ and $L_{y_H}$ coincide on $(y_H,F(q))$, implying that $(F(d),0)$ is above the line $L_{y_H}$.

 For case $(ii)$ in Lemma \ref{second-order prop-H},  note that if $y_\theta>y_H$, the line $L_{y_H}$ coincides with the line $L_{y_\theta}(y)=H'(y_\theta^+)(y-y_\theta)+H(y_\theta)$, because the graph of $H(\cdot)$ on $y\in (y_H, y_\theta)$ is a line segment.  The strict convexity of $H(\cdot)$ on $(y_\theta, F(q))$ now implies $(F(q),0)$ is above $L_{y_H}$ and so is  $(F(d),0)$.

  Therefore,  $(F(d),0) \in T_H$ by Lemma \ref{lemma-TH=XH}, and thus there exists some $z_0 \in [y_h,y_H]$ such that the point $(F(d),0)$ is on the tangent line $L_{z_0}$. This implies the existence of \eqref{equation for unique z_0}. Finally, by the definition of $y_H$,  $H(\cdot)$ is strictly concave on $(y_h,y_H)$, which implies the uniqueness of the solution.
\end{proof}

The following theorem is the main result in this section.

\begin{theorem}\label{value function for barrier put option with American feature}
The solution to  (\ref{vf}) is given by
\begin{equation}\label{V(x)}
V(x)=\varphi(x)W(F(x))=
\begin{cases}
q-x, & \mbox{if} \  0<x \leq x_0, \\
(q-x_0) \cdot \dfrac{\varphi(x)}{\varphi(x_0)} \cdot \dfrac{F(d)-F(x)}{F(d)-F(x_0)}, & \mbox{if} \  x_0 < x < d,
\end{cases}
\end{equation}
where $x_0$ is determined by
\begin{equation}\label{x_0 equation}
\dfrac{(h'\psi-h\psi')(x_0)}{F(d)}=(h'\varphi-h\varphi')(x_0).
\end{equation}
The optimal stopping time is given by
\begin{align}\label{optimal stopping time}
\tau^\star =\inf \lbrace t\geq 0: X_t \leq x_0 \rbrace.
\end{align}
\end{theorem}

\begin{proof}
Firstly, note that Lemma \ref{uniqueness z_0} guarantees the existence of a unique $z_0 \in [y_h, y_H]$ such that 
\begin{equation}\label{equation for z_0}
H'(z_0)=\dfrac{H(F(d))-H(z_0)}{F(d)-z_0}=\dfrac{-H(z_0)}{F(d)-z_0}
\end{equation}
Therefore, the straight line $L_{z_0}: [0,F(d)] \rightarrow \mathbb{R}$,
\begin{align*}
L_{z_0}(y) \triangleq H(z_0)+H'(z_0)(y-z_0), \ \ \ y \in [0,F(d)]
\end{align*}
is tangent to $H$ at $z_0$ and coincides with the chord expanding between $(z_0, H(z_0))$ and $(F(d),0)$ over the graph of $H$. Since $H(z_0)>0$, (\ref{equation for z_0}) implies that $L_{z_0}$ is decreasing.  We also point out that $H(\cdot)$ is concave on $(0,z_0)$ since $z_0\leq y_H\leq y_\theta$.

Define $W(\cdot)$ as follows,
\begin{align}\label{W1}
W(y)=
\begin{cases}
H(y), & \mbox{if} \ y\in (0,z_0], \\
L_{z_0}(y)=H(z_0)\dfrac{F(d)-y}{F(d)-z_0}, & \mbox{if} \ y\in (z_0,F(d)).
\end{cases}
\end{align}
It is clear that $W$ is concave and dominates $H$ on $(0,F(d))$.We claim that $W$ is the smallest non-negative concave majorant of $H$ on $(0,F(d))$. 
Indeed, $W$ is obviously the smallest concave majorant of $H$ on $(0,z_0]$, and the straight line $L_{z_0}$ is the smallest concave curve that connects $(z_0,H(z_0))$ with $(F(d),0)$.

Let $x\triangleq F^{-1}(y)$, $y\in [0,F(d)]$, and $x_0 \triangleq F^{-1}(z_0)\in (0,q)$. Recalling that $H(y)=\left( \dfrac{h}{\varphi} \right) \circ F^{-1}(y)$ for $y>0$ by \eqref{H}, by \cite[Proposition 5.5]{DK03}, we have
 \begin{equation}\label{V1}
V(x)=\varphi(x)W(F(x))=
\begin{cases}
q-x, & \mbox{if} \  x\in(0,x_0 ]\\
(q-x_0).\dfrac{\varphi(x)}{\varphi(x_0)}. \dfrac{F(d)-F(x)}{F(d)-F(x_0)}, & \mbox{if} \  x\in (x_0,d)
\end{cases}
\end{equation}

Let $\Gamma \triangleq \lbrace x\in (0,d): V(x)=h(x)  \rbrace$ be the stopping region, then the stopping time $\tau^\star \triangleq \inf \lbrace t\geq 0: X_t \in \Gamma \rbrace$ is optimal by \cite[Proposition 5.7]{DK03}, noting that $\lim\limits_{x\to b}\frac{h(x)}{\psi(x)}=0$ and $h$ is continuous. Now we identify the stopping region $\Gamma$.

Define \begin{align*}
\tilde{C} \triangleq \lbrace y \in [0,F(d)]: W(y)>H(y) \rbrace
\end{align*}
and then $\tilde{C}=(z_0, F(d)),$ since $L_{z_0}$ dominates $H(\cdot)$ on $(z_0,F(d))$. The continuation region is
\begin{align*}
C \triangleq \lbrace x\in (0,d): V(x)>h(x) \rbrace=F^{-1}(\tilde{C})=F^{-1}((z_0,F(d)))=(x_0,d).
\end{align*}
Therefore, $\Gamma=C^C=(0,x_0],$ and hence 
\begin{equation}\label{OPST}
\tau^\star =\inf \lbrace t\geq 0: X_t \in \Gamma \rbrace =\inf \lbrace t\geq 0: X_t \leq x_0) \rbrace,
\end{equation}
where $x_0$ is characterized by
\begin{equation}\label{x_0}
\begin{cases}
H'(z_0)=\dfrac{-H(z_0)}{F(d)-z_0},\\
x_0=F^{-1}(z_0),
\end{cases}
\end{equation}
which is equivalent to \eqref{x_0 equation}, by 
equations \eqref{H}, \eqref{F}, \eqref{W} and \eqref{H'}.  The proof is concluded.
\end{proof}

Based on Theorem \ref{value function for barrier put option with American feature}, by letting the barrier $d$ go to infinity, we can derive the value function and optimal stopping time for a standard American put option. 

\begin{corollary}\label{value function for standard American put option}
The following optimal stopping problem
\begin{equation*}
V(x) \triangleq \sup_{\tau \geq 0} E_x[e^{-r\tau}(q-X_\tau)^+I_{\lbrace \tau<\infty \rbrace }],
\end{equation*}
admits the solution
\begin{equation}\label{apo}
V(x)=\varphi(x)W(F(x))=
\begin{cases}
q-x, & \mbox{if} \  0<x \leq x_0 \\(q-x_0)
\dfrac{\varphi(x)}{\varphi(x_0)}, & \mbox{if} \  x>x_0
\end{cases}
\end{equation}
where the optimal price level $x_0$ is uniquely found from the following equation:
\begin{equation}\label{x_0 simple equation}
\dfrac{h'(x_0)}{\varphi'(x_0)}=\dfrac{h(x_0)}{\varphi(x_0)}\,,
\end{equation}
that is, $\varphi(x_0)+(q-x_0)\cdot \varphi'(x_0)=0$.
The optimal stopping time is given by
\begin{align*}
\tau^\star =\inf \lbrace t\geq 0: X_t \leq x_0 \rbrace.
\end{align*}
\end{corollary}

\begin{proof}
Let $V_d(x):= \sup_{\tau \geq 0} E_x[e^{-r\tau} (q-X_{\tau})^+I_{\lbrace \tau<\tau_d \rbrace }]$ and $V_{\infty}(x):= \sup_{\tau \geq 0} E_x[e^{-r\tau} (q-X_{\tau})^+I_{\lbrace \tau<\infty \rbrace }]$. Note that 
\begin{align*}
0\le V_\infty(x)-V_d(x)\le \sup_{\tau\ge 0}E_x[e^{-r\tau}(q-X_{\tau})^+I_{\lbrace \tau_d\le  \tau<\infty \rbrace }]\le q E_x[I_{\lbrace \tau_d<\infty \rbrace }],
\end{align*}
where the last term goes to zero as $d\to +\infty$ by dominated convergence theorem. Consequently, we have $\lim\limits_{d\to+\infty}V_d(x)=V_\infty(x)$. 

To finish the proof, it suffices to show $x_0\triangleq\lim\limits_{d\to+\infty} x_0^d$ exists and satisfies \eqref{x_0 simple equation}, where $x_0^d$ is the unique solution of $$\dfrac{(h'\psi-h\psi')(x)}{F(d)}=(h'\varphi-h\varphi')(x).$$ Note that $x_0^d$  decreases as $d$ increases, by Proposition \ref{external comparison principles for pricing barrier option with American feature}. Then $x_0$ exists, and it satisfies \eqref{x_0 simple equation}, noting that $ \sup\limits_{x:F(x)\in [y_h,y_H] } |(h'\psi-h\psi')(x)|<+\infty$ and $\lim\limits_{d\to+\infty}F(d)=+\infty$.   The uniqueness of the solution to \eqref{x_0 simple equation} follows from the fact that  $(h'\varphi-h\varphi')(x)$ is strictly decreasing. 
\end{proof}

The following proposition provides the regularity property of the value function.

\begin{proposition}
The value function $V$ in \eqref{V(x)} belongs to $C^1\left( (0, d) \right) \cap C^2 \left( (0, d) \setminus \{ x_0 \} \right)$,  where $x_0$ is given by \eqref{x_0 equation}. In particular, the smooth-fit principle holds for the value function.
\end{proposition}

\begin{proof}
Without any further difficulty, the alleged results can be established by the same argument used in Proposition \ref{smooth-fit}.
\end{proof}

The following propositions are analogues of Proposition \ref{external comparison principles under ESO problem} and Proposition \ref{internal comparison principles under ESO problem}, respectively. 

\begin{proposition}\label{external comparison principles for pricing barrier option with American feature}
The optimal price level $x_0$ given by the equation (\ref{x_0 equation}) is decreasing with respect to $d$, while increasing with respect to $r$ and $q$. 
\end{proposition}

\begin{proof}
Denote the value function $V(x)$ in (\ref{vf}) by $V_d(x)$ to emphasize its dependence on $d$, and we shall use similar representations for the other parameters. By \eqref{vf}, it is easy to see
\begin{equation*}
V_{r_1}(x)>V_{r_2}(x) \ \mbox{for} \ r_1<r_2; \ \ V_{q_1}(x)<V_{q_2}(x) \ \mbox{ for} \ q_1<q_2; \ \ V_{d_1}(x)<V_{d_2}(x) \ \mbox{ for} \ d_1<d_2.
\end{equation*}
We denote by $C_r$, $C_d$ and $C_q$, the continuation regions for $V_{r}(x)$, $V_{d}(x)$ and $V_{q}(x)$, respectively. Then it follows from the above inequalities that 
\begin{equation} \label{C_r}
C_{r_2}=\lbrace x: V_{r_2}(x)>h(x) \rbrace \subset \lbrace x: V_{r_1}(x)>h(x) \rbrace = C_{r_1},
\end{equation}
\begin{equation}\label{C_d}
C_{d_1}=\lbrace x: V_{d_1}(x)>h(x) \rbrace \subset \lbrace x: V_{d_2}(x)>h(x) \rbrace = C_{d_2}.
\end{equation}
By Theorem \ref{value function for barrier put option with American feature}, (\ref{C_r}) and (\ref{C_d}), we have $x_0^{r_1}\le x_0^{r_2}$ and $x_0^{d_2}\le x_0^{d_1}$.

Finally, we show $x_0^{q_1}\leq x_0^{q_2}$ for $q_1<q_2$. Assuming, on the contrary, that $x_0^{q_1}> x_0^{q_2}$, and taking some $x \in (x_0^{q_2}, x_0^{q_1})$, we have
\begin{equation*}
V_{q_2}(x)> q_2-x \ \ \mbox{and} \ \ V_{q_1}(x)=q_1-x,
\end{equation*}
since such $x \in C_{q_2}$ but $x \notin C_{q_1}$. As a result, we derive that $V_{q_2}(x)-V_{q_1}(x)>q_2-q_1$. However, this is in contradiction with the following fact,
\begin{align*}
V_{q_2}(x)&=\sup_{\tau \geq 0} E_x[e^{-r\tau} (q_2-X_{\tau})^+I_{\lbrace \tau<\tau_d \rbrace }]\\
& \leq \sup_{\tau \geq 0} E_x[e^{-r\tau} \left((q_1-X_{\tau})^++(q_2-q_1)\right)I_{\lbrace \tau<\tau_d \rbrace }]\\
&\leq \sup_{\tau \geq 0} E_x[e^{-r\tau} (q_1-X_{\tau})^+I_{\lbrace \tau<\tau_d \rbrace }] +q_2-q_1\\
&= V_{q_1}(x)+q_2-q_1.
\end{align*}
\end{proof}
%

\begin{proposition}\label{internal comparison principles for pricing barrier option with American feature}
The value function $V(x)$ given in \eqref{V(x)} (and \eqref{apo}) is non-increasing  with respect to the drift  $\mu$. That is, $V_{\mu_1}(x)\geq V_{\mu_2}(x)$, if $\mu_1(x) \leq \mu_2(x)$ for all $x\in (0,d)$. The optimal price level $x_0$ is non-decreasing with respect to $\mu$. 
\end{proposition}

\begin{proof} We shall prove the result for \eqref{V(x)}, and the proof for \eqref{apo} can be done in a similar way. The proof is similar to the proof of Proposition \ref{internal comparison principles under ESO problem}.

Denote by $\Gamma_\mu \ (\mbox{resp.} \ C_\mu)$ the stopping region (resp. continuation region) of $V_\mu$. Note that $V_{\mu_2}(x)\leq V_{\mu_1}(x)$ for $x\in \Gamma_{\mu_2}$ since immediate stop policy is the smallest possible optimal candidate. Thus we only need to show $V_{\mu_2}(x)\leq V_{\mu_1}(x)$ for  $x\in C_{\mu_2}$. 
Denote by $\mathcal{A}_{\mu}$ the differential operator
$\mathcal{A}_{\mu} \triangleq \mathcal{L}_{\mu} -r$ where $\mathcal{L}_{\mu}$ is the infinitesimal generator of the diffusion $X^\mu$. Along the lines of the notation, we denote by $(\psi_{\mu},\varphi_{\mu})$, the pair of fundamental solutions of $(\mathcal{A}_{\mu}u)(x)=0$. As usual, $\psi_{\mu}(x)$ is strictly increasing whilst $\varphi_{\mu}(x)$ is strictly decreasing. 

We define $\tau_n^{\mu_i}\triangleq n\wedge \tau_y^{\mu_i}$ with $\tau_y^{\mu_i}\triangleq \inf \lbrace t\geq 0: X_t^{\mu_i}=y \rbrace$. Since $\lbrace \tau_n^{\mu_i} \rbrace_{n\geq 1}$ is a sequence of almost surely finite stopping times, a direct use of Dynkin's formula implies  that, for all $x\geq y$, 
\begin{align*}
E_x\left[ e^{-r\tau_n^{\mu_2}} V_{\mu_1}( X_{\tau_n^{\mu_2}}^{\mu_2}) \right] \ &= \ V_{\mu_1}(x)+ E_x {\int}_0^{\tau_n^{\mu_2}}e^{-rt}\left(\mathcal{A}_{\mu_2}V_{\mu_1}\right)\left(X_t^{\mu_2}\right)dt  \\
\ &\leq \ V_{\mu_1}(x)+ E_x {\int}_0^{\tau_n^{\mu_2}}e^{-rt}\left(\left(\mathcal{A}_{\mu_2}-\mathcal{A}_{\mu_1}\right)V_{\mu_1}\right)\left(X_t^{\mu_2}\right)dt \\ 
\ &= \ V_{\mu_1}(x)+ E_x {\int}_0^{\tau_n^{\mu_2}}e^{-rt}\left(\left(\mu_2-\mu_1\right)V'_{\mu_1}\left(X_t^{\mu_2}\right)\right)dt 
\\
\ &\leq \ V_{\mu_1}(x)
\end{align*}
since, by \eqref{V(x)}, $(\mathcal{A}_{\mu_1}V_{\mu_1})(x) \leq 0$ and $V'_{\mu_1}(x)< 0$ for all $x\in (0,d)$. As a consequence, 
\begin{align*}
V_{\mu_1}(x) \ &\geq \ E_x\left[ e^{-r\tau_y^{\mu_2}} V_{\mu_1}( X_{\tau_y^{\mu_2}}^{\mu_2}); \tau_y^{\mu_2} < n \right] + E_x\left[ e^{-rn} V_{\mu_1}( X_n^{\mu_2}); \tau_y^{\mu_2} \geq n \right] \\
&\geq \ E_x\left[ e^{-r\tau_y^{\mu_2}} V_{\mu_1}( X_{\tau_y^{\mu_2}}^{\mu_2}); \tau_y^{\mu_2} < n \right] \\
&= \ V_{\mu_1}(y)E_x\left[ e^{-r\tau_y^{\mu_2}}; \tau_y^{\mu_2} < n \right] 
\end{align*}
Letting $n \rightarrow +\infty$, by \eqref{eq-st}, for all $x\geq y$,  
\begin{align*}
\dfrac{\varphi_{\mu_2}(x)}{\varphi_{\mu_2}(y)}=E_x\left[ e^{-r\tau_y^{\mu_2}}; \tau_y^{\mu_2} < +\infty \right]  \leq \dfrac{V_{\mu_1}(x)}{V_{\mu_1}(y)}
\end{align*}
Note that on the continuation region $C_{\mu_i}$, by \eqref{V(x)}, $V_{\mu_i}(\cdot)$ can be expressed as 
\begin{equation*}
V_{\mu_i}(\cdot) \ = \ c_1^{\mu_i}\varphi_{\mu_i}(\cdot) + c_2^{\mu_i}\psi_{\mu_i}(\cdot)
\end{equation*}
with constants $c_1^{\mu_i}>0$ and $c_2^{\mu_i}<0$. It is then easy to check that for $x\geq y \in C_{\mu_i}$,
\begin{align*}
\dfrac{V_{\mu_i}(x)}{V_{\mu_i}(y)} \ = \ \dfrac{c_1^{\mu_i}\varphi_{\mu_i}(x)+ c_2^{\mu_i}\psi_{\mu_i}(x)}{c_1^{\mu_i}\varphi_{\mu_i}(y)+ c_2^{\mu_i}\psi_{\mu_i}(y)} \ \leq \ \dfrac{\varphi_{\mu_i}(x)}{\varphi_{\mu_i}(y)}
\end{align*}
due to $\varphi_{\mu_i}(x)\psi_{\mu_i}(y) \leq \varphi_{\mu_i}(y)\psi_{\mu_i}(x)$ for $x\geq y$. Consequently, we see that for all $x\geq y \in C_{\mu_2}$,
\begin{align}\label{FHT}
\dfrac{V_{\mu_2}(x)}{V_{\mu_2}(y)} \ \leq \ \dfrac{\varphi_{\mu_2}(x)}{\varphi_{\mu_2}(y)} \ \leq \ \dfrac{V_{\mu_1}(x)}{V_{\mu_1}(y)} 
\end{align}
Adjusting the terms, we obtain that $\dfrac{V_{\mu_1}(y)}{V_{\mu_2}(y)}\leq \dfrac{V_{\mu_1}(x)}{V_{\mu_2}(x)}$ for all $x\geq y \in C_{\mu_2}$. That is, $\dfrac{V_{\mu_1}}{V_{\mu_2}}(\cdot)$ is increasing on $C_{\mu_2}$. Therefore, we can derive by taking $y \rightarrow x_0^{\mu_2}$ that 
\begin{align*}
\dfrac{V_{\mu_1}(x)}{V_{\mu_2}(x)} \  \geq \ \dfrac{V_{\mu_1}(x_0^{\mu_2})}{V_{\mu_2}(x_0^{\mu_2})} \ \geq \ 1
\end{align*}
for all $x \in C_{\mu_2}$.
As a consequence, we obtain that $V_{\mu_2}(x)\leq V_{\mu_1}(x)$ on $x \in C_{\mu_2}$. Therefore,
\begin{equation} \label{C_mu}
C_{\mu_2}=\lbrace x: V_{\mu_2}(x)>h(x) \rbrace \subset \lbrace x: V_{\mu_1}(x)>h(x) \rbrace = C_{\mu_1},
\end{equation}
and this implies $x_0^{\mu_1}\le x_0^{\mu_2}$.  \end{proof}

\begin{remark}
The financial implication of (\ref{C_mu}) is quite clear. That is, increasing the value of $\mu$ (drift factor) decreases the value function and shrinks the continuation region where wait is optimal, and thus accelerates rational exercise.  In terms of the dividend yield, the above proposition indicates that a higher dividend yield leads to a higher option premium of an American barrier put option.
\end{remark}

For a standard American put option (without a barrier), the following proposition indicates the dependence of the value function and the optimal exercise boundary on the volatility $\sigma$. 

\begin{proposition}\label{internal comparison principles for pricing standard American put option} 
The value function $V(x)$ given in \eqref{apo} is non-decreasing  with respect to the volatility  $\sigma$. That is, if  $\sigma_1(x)\leq \sigma_2(x)$ for all $x\in (0,+\infty)$, then $V_{\sigma_1}(x)\leq V_{\sigma_2}(x)$ for all $x\in (0,+\infty)$.  Consequently, the optimal price level $x_0$ is non-increasing with respect to $\sigma$. 
\end{proposition}

\begin{proof} We shall follow the idea used in the proofs of Propositions \ref{internal comparison principles under ESO problem} and \ref{internal comparison principles for pricing barrier option with American feature}.

Denote by $\Gamma_\sigma \ (\mbox{resp.} \ C_\sigma)$ the stopping region (resp. the continuation region) of $V_\sigma$. Note that $V_{\sigma_1}(x)\leq V_{\sigma_2}(x)$ on $\Gamma_{\sigma_1}$ since immediate exercise yields the least value. Thus we only need to show $V_{\sigma_1}(x)\leq V_{\sigma_2}(x)$ on $C_{\sigma_1}$.

Denote by $\mathcal{A}_{\sigma}$ the differential operator
$\mathcal{A}_{\sigma} \triangleq \mathcal{L}_{\sigma} -r$,
where $\mathcal{L}_{\sigma}$ is the infinitesimal generator of the diffusion $X^\sigma$. Similar as in the proof of Proposition \ref{internal comparison principles for pricing barrier option with American feature}, we denote by $(\psi_{\sigma},\varphi_{\sigma})$, the pair of fundamental solutions of $(\mathcal{A}_{\sigma}u)(x)=0$, where $\psi_{\sigma}(x)$ is strictly increasing whilst $\varphi_{\sigma}(x)$ is strictly decreasing. 

Define $\tau_n^{\sigma_i}:= n\wedge \tau_y^{\sigma_i}$ where $\tau_y^{\sigma_i}:= \inf \lbrace t\geq 0: X_t^{\sigma_i}=y \rbrace$. By Dynkin's formula, for all $x\geq y$, 
\begin{align*}
E_x\left[ e^{-r\tau_n^{\sigma_1}} \varphi_{\sigma_2}( X_{\tau_n^{\sigma_1}}^{\sigma_1}) \right] \ &= \ \varphi_{\sigma_2}(x)+ E_x {\int}_0^{\tau_n^{\sigma_1}}e^{-rt}\left(\mathcal{A}_{\sigma_1} \varphi_{\sigma_2}\right)\left(X_t^{\sigma_1}\right)dt  \\
\ &= \ \varphi_{\sigma_2}(x)+  E_x {\int}_0^{\tau_n^{\sigma_1}}e^{-rt}\left(\left(\mathcal{A}_{\sigma_1}-\mathcal{A}_{\sigma_2}\right) \varphi_{\sigma_2}\right)\left(X_t^{\sigma_1}\right)dt \\ 
\ &= \ \varphi_{\sigma_2}(x)+  E_x {\int}_0^{\tau_n^{\sigma_1}}e^{-rt}\left(\dfrac{1}{2} \left(\sigma_1^2-\sigma_2^2\right) \varphi''_{\sigma_2}\left(X_t^{\sigma_1}\right)\right)dt 
\\
\ &\leq \ \varphi_{\sigma_2}(x)
\end{align*}
since $\mathcal{A}_{\sigma_2}\varphi_{\sigma_2}(x) \equiv 0$ and $\varphi''_{\sigma_2}(x)>0$ for all $x$. Letting $n$ go to infinity, by \eqref{eq-st}, we have for all $x\geq y$,
\begin{align*}
\dfrac{\varphi_{\sigma_1}(x)}{\varphi_{\sigma_1}(y)} \leq \dfrac{\varphi_{\sigma_2}(x)}{\varphi_{\sigma_2}(y)}.
\end{align*}
Hence, by \eqref{apo}, for $x \in C_{\sigma_1}$,
\begin{align}\label{eq-55}
V_{\sigma_1}(x) = \ h(x_0^{\sigma_1})\dfrac{\varphi_{\sigma_1}(x)}{\varphi_{\sigma_1}(x_0^{\sigma_1})} \ \leq \ h(x_0^{\sigma_1})\dfrac{\varphi_{\sigma_2}(x)}{\varphi_{\sigma_2}(x_0^{\sigma_1})}.
\end{align}
On the other hand, for $x \in C_{\sigma_1}$,
\begin{align}\label{eq-56}
V_{\sigma_2}(x) \ &\geq \ E_x\left[e^{-r\tau_{x_0^{\sigma_1}}^{\sigma_2}} h\left(X_{\tau_{x_0^{\sigma_1}}^{\sigma_2}}^{\sigma_2}\right)\right] \ = \ h(x_0^{\sigma_1})\dfrac{\varphi_{\sigma_2}(x)}{\varphi_{\sigma_2}(x_0^{\sigma_1})}.
\end{align}
Combine \eqref{eq-55} and \eqref{eq-56} and we have $V_{\sigma_1}(x)\leq V_{\sigma_2}(x)$ for $x \in C_{\sigma_1}$. Thus $V_{\sigma_1}(x)\leq V_{\sigma_2}(x)$ for $x \in C_{\sigma_1}$ for all $x\in (0,+\infty)$.

Finally, observing  
\begin{equation*} \label{C_sigma}
(x_0^{\sigma_1}, +\infty)=C_{\sigma_1}=\lbrace x: V_{\sigma_1}(x)>h(x) \rbrace \subset \lbrace x: V_{\sigma_2}(x)>h(x) \rbrace = C_{\sigma_2}=(x_0^{\sigma_2}, +\infty),
\end{equation*}
we have $x_0^{\sigma_2}\le x_0^{\sigma_1}$. The proof is concluded.
\end{proof}

\begin{tabular}{lll}
Dongchao Huang\\
Department of Mathematics
University of Hong Kong\\
Hong Kong\\
{\tt huangdongchao2012@gmail.com}
\end{tabular}
\\

\begin{tabular}{lll}
Jian Song \\
Department of Mathematics and Department of Statistics \& Actuarial Science\\
University of Hong Kong\\
Hong Kong\\
{\tt txjsong@hku.hk}
\end{tabular}


\begin{thebibliography}{99}
\bibitem{AL03}  Alvarez, L.H.R., \newblock On the properties of r-excessive mappings for a class of diffusions. \newblock {\em The Annals of Applied Probability}  2003, Vol. 13, No. 4, 1517-1533.

\bibitem{BS02} Borodin, A. N., and Salminen, P.. (2002) \newblock {\em Handbook of Brownian Motion-Facts and Formulae}, \newblock 2nd edn.  \newblock Basel: Birkh\"auser.

\bibitem{Dai04} Dai, Min, and Yue Kuen Kwok.\newblock Knock in American options.\newblock {\em Journal of Futures Markets}, 24 (2004), no. 2, 179-192.


\bibitem{DK03}  Dayanik, S. and Karatzas, I. \newblock
On the optimal stopping problem for one-dimensional diffusions.
\newblock {\em Stochastic Process. Appl.} 107 (2003), no. 2, 173-212.

\bibitem{E04} Ekstr\:om, E. Properties of American options prices. \newblock{\em Stochastic Processes and Their Applications.} 114, 265-278, (2004).

\bibitem{Gao00}
Gao, Bin, Jing-zhi Huang, and Marti Subrahmanyam.  \newblock  The valuation of American barrier options using the decomposition technique.  \newblock {\em Journal of Economic Dynamics and Control} \newblock 24, no.11 (2000): 1783-1827.


\bibitem{XG01} Guo, X. and  Shepp, L. \newblock Some optimal stopping problems with nontrivial boundaries for pricing exotic options. \newblock {\em J. Appl. Prob.} 38, 647-658 (2001).


\bibitem{Henderson05} Henderson, V. Analytical comparisons of option prices in stochastic volatility models. \newblock{\em Mathematical Finance.} Vol. 15, No. 1 49-59 (2005).

\bibitem{Hobson10} Hobson, D. Comparison results for stochastic volatility models via coupling. \newblock {\em Finance Stoch.} 14: 129- 152 (2010).

\bibitem{IM74}
It\^o, K., McKean Jr., H.P. \newblock{\em Diffusion Processes and Their Sample Paths.}  Springer, Berlin, 1974.

\bibitem{JT03} Janson, S. and Tysk, J. Volatility time and properties of option prices. \newblock{\em the Annals of Applied Probability}. Vol. 13, No. 3, 890-913 (2003).

\bibitem{Jun13}
Jun, D. and Ku, H. \newblock Valuation of American partial barrier options. \newblock{\em Review of Derivatives Research} (2013): 1-25.

\bibitem{Jun15}
Jun, D. and Ku, H. \newblock Analytic solution for American barrier options with two barriers. . \newblock {\em Journal of Mathematical Analysis and Applications}, (2015), 422(1), 408-423.



\bibitem{KS} Karatzas, I. and Shreve, S. E. \newblock{\em Brownian motion and stochastic calculus.} Springer-Verlag. Second edition.  


\bibitem{KS00} Karatzas, I. and Wang, H. \newblock A barrier option of American type. \newblock {\em Applied Mathematics $\&$ Optimization}, 42, no.3 (2000): 259-279.

\bibitem{LM10} Linetsky V, Mendoza R. Constant elasticity of variance (CEV) diffusion model.  \newblock{\em Encyclopedia Quant. Finance.} (2010)


\bibitem{Ott13} Ott, Curdin. \newblock Optimal stopping problems for the maximum process with upper and lower caps. \newblock {\em The Annals of Applied Probability} 23, no. 6 (2013): 2327-2356.

\bibitem{Peskir06}  Peskir, Goran, and Albert Shiryaev. \newblock Optimal stopping and free-boundary problems, \newblock Springer Science $\&$ Business Media, 2006.
\bibitem{Pham09} Pham, Huy\^en. \newblock Continuous-time stochastic control and optimization with financial applications. \newblock Vol. 61. \newblock Springer Science $\&$ Business Media, 2009.
\bibitem{DW1987} Rogers, L. C. G., and Williams, David. \newblock Diffusions, Markov Processes, and Martingales.\newblock Vol. 2. {\em It\^o Calculus.} \newblock New York: Wiley, 1987.
\bibitem{Shiryaev07} Shiryaev, Albert N. \newblock Optimal stopping rules. \newblock Vol. 8. \newblock Springer Science $\&$ Business Media, 2007.

\bibitem{Yin12} Yin, Hong-Ming. Some properties for the American option-pricing model. \newblock {\em Journal of Mathematical Finance,} 2012, 2, 243-250. 

\end{thebibliography}
\end{document}